\documentclass[10pt]{article}
\usepackage[usenames]{color} %used for font color
\usepackage{amssymb} %maths
\usepackage{amsmath} %maths
\usepackage[utf8]{inputenc} %useful to type directly diacritic characters
\usepackage[utf8]{inputenc}
\usepackage[T1]{fontenc}
\usepackage[toc,page]{appendix}
\usepackage{geometry}
\usepackage{fancyhdr}
\usepackage{setspace}
\usepackage{amsmath}
\usepackage{amsthm}
\usepackage{latexsym}
\usepackage{amssymb}
\usepackage{amsfonts}
\usepackage{pifont}
\usepackage{epstopdf}
\usepackage{graphicx}
\usepackage{tensor}
\usepackage{tikz}
\usepackage{relsize}
\usepackage{url}
\usepackage{hyperref}
\usepackage{color}
\let\headruleORIG\headrule
\renewcommand{\headrule}{\color{black} \headruleORIG}

\usepackage{colortbl}
\arrayrulecolor{black}

\fancypagestyle{plain}{
  \fancyhead{}
  \fancyfoot[C]{\thepage}
  
}
\numberwithin{equation}{section}
\usepackage{nomencl}
\makenomenclature

\title{Static spherically symmetric Einstein-Vlasov bifurcations of the Schwarzschild spacetime}
\author{Fatima Ezzahra Jabiri \thanks{Sorbonne Universit\'{e}, CNRS, Universit\'{e} de Paris, Laboratoire Jacques-Louis Lions (LJLL), F-75005 Paris, France, jabiri@ljll.math.upmc.fr} }
\theoremstyle{plain}
\newtheorem{remark}{Remark}
\theoremstyle{plain}
\newtheorem{theoreme}{Theorem}
\theoremstyle{definition}
\newtheorem{definition}{Definition}
\theoremstyle{plain}
\newtheorem{lemma}{Lemma}
\theoremstyle{plain}
\newtheorem{Propo}{Proposition}
\theoremstyle{plain}

\theoremstyle{plain}

\theoremstyle{definition}
\newtheorem*{thank}{Acknowledgements}

\newcommand{\Ric}[2]{\tensor{Ric}{_{#1#2}}}

\newcommand{\g}[2]{\tensor{g}{_{#1#2}}}
\newcommand{\ginv}[2]{\tensor{g}{^{#1#2}}}
\newcommand{\T}[2]{\tensor{{T}}{_{#1#2}}}
\newcommand{\Ein}[2]{\tensor{G}{_{#1#2}}}
\newcommand{\Chris}[3]{\tensor{\Gamma}{^{#1}_{#2#3}}}

\newcommand{\spacetime}{\mathcal{M}}

\newcommand{\Esch}{E^{Sch}_\ell}
\newcommand{\ELrange}{\left]\sqrt{\frac{8}{9}}, \infty\right[\times \left]12M^2, \infty\right[}
\newcommand{\Abound}{\mathcal A_{bound}}

\newcommand{\Chi}{\mbox{\Large$\chi$}}
\newcommand{\Ibarre}{\overline{I}}
\newcommand{\tmu}{\tilde\mu}

\makeindex
\begin{document}

\pagestyle{fancy}

\maketitle

\begin{abstract}
We construct a one-parameter family of static and spherically symmetric solutions to the Einstein-Vlasov system bifurcating from the Schwarzschild spacetime. The constructed solutions have the property that the spatial support of the matter  is a finite, spherically symmetric shell located away from the black hole. Our proof is mostly based on the analysis of the set of trapped timelike geodesics and of the effective potential energy for static spacetimes close to Schwarzschild. This provides an alternative approach to the construction of static solutions to the Einstein-Vlasov system in a neighbourhood of black hole vacuum spacetimes. 

\end{abstract}

\setcounter{secnumdepth}{4}
\setcounter{tocdepth}{4}

%régler l'espacement entre les lignes
\newcommand{\HRule}{\rule{\linewidth}{0.5mm}}
%espacement entre les lignes d'un tableau
\renewcommand{\arraystretch}{1.5}

\tableofcontents

\section{Introduction}

The relativistic kinetic theory of gases plays an important role in the description of matter fields surrounding black holes such as plasmas or distribution of stars \cite{peebles1972star}, \cite{bahcall1976star}. In particular, Vlasov matter is used to analyse galaxies or globular galaxies where the stars play the role of gas particles and where collisions between these are sufficiently rare to be neglected, so that the only interaction taken into account is gravitation. 
\\
\\ In the geometric context of general relativity, the formulation of relativistic kinetic theory  was developed by Synge  \cite{synge1934energy}, Israel \cite{israel1963relativistic} and Tauber and Weinberg \cite{tauber1961internal}. We also refer to  the introduction of \cite{andersson2019variational} for more details. In this work, we are interested in the Vlasov matter model. It is described by a particle distribution function on the phase space which is transported along causal geodesics, resulting in the Vlasov equation. The latter is coupled to the equations for the gravitational field, where the source terms are computed from the distribution function. In the non-relativistic setting, i.e.~the Newtonian framework, the resulting nonlinear system of partial diffrential equations is the Vlasov-Poisson (VP) system \cite{rein1992newtonian}, while its general relativistic counterpart forms the Einstein-Vlasov (EV) system. Collisionless matter possesses several attractive features  from a partial differential equations viewpoint. On any fixed background, it avoids pathologies such as shock formation, contrary to more traditional fluid models. Moreover, one has global solutions of the VP system in three dimensions for general initial data \cite{pfaffelmoser1992global}, \cite{lions1991propagation}. 
\\
\\ The local well-posedness of the Cauchy problem for the EV system was first investigated in \cite{choquet1971probleme} by Choquet-Bruhat. Concerning the nonlinear stability of the Minkowski spacetime as the trivial solution of the EV system, it was proven in the case of spherically symmetric initial data by Rendall and Rein \cite{rein1992global} in the massive case and by Dafermos \cite{dafermos2006note} for the massless case.  The general case was recently shown  by Fajman, Joudioux and Smulevici \cite{fajman2017stability} and independently by  Lindblad-Taylor \cite{lindblad2017global} for the massive case, and by Taylor \cite{taylor2017global} for the massless case. Nonlinear stability results have been given by Fajman \cite{fajman2017nonvacuum}  and Ringtröm \cite{ringstrom2013topology} in the case of cosmological spacetimes. See also \cite{smulevici2011area}, \cite{andreasson2005existence}, \cite{dafermos2006strong}, \cite{weaver2004area}, \cite{smulevici2008strong} for several results on cosmological spacetimes with symmetries.  
\\
\\In stellar dynamics, equilibrium states can be described by stationary solutions of the VP or EV system. Static and spherically symmetric solutions can be obtained by assuming that the distribution function has the following form 
\begin{equation*}
f(x,v) = \Phi(E, \ell),
\end{equation*}
where $E$ and $\ell$ are interpreted as the energy and the total angular momentum of particles respectively. In fact, in the Newtonian setting, the distribution function associated to a stationary and spherically symmetric solution to the VP system is necessarily described by a function depending only on $E$ and $\ell$.  Such statement is referred to as Jean's theorem \cite{jeans2017problems}, \cite{jeans1915theory}, \cite{batt1986stationary}. However, it has been shown that its generalisation to general relativity is false in general \cite{schaeffer1999class}.  
\\ A particular choice of $\Phi$, called the polytropic ansatz, which is commonly used to construct static and spherically symmetric states for both VP and EV system is given by 
\begin{equation} 
\label{polytropes}
\Phi(E, \ell):=
\left\{ 
\begin{aligned}
(E_0 - E)^\mu \ell^k, \quad &E<E_0, \\
0 ,\quad &E\geq E_0. 
\end{aligned}
\right. 
\end{equation}
where $E_0>0$, $\mu>-1$ and $k>-1$.
In \cite{rein1993smooth}, Rein and Rendall gave the first class of asymptotically flat, static, spherically symmetric solutions to EV system with finite mass and finite support such that $\Phi$ depends only on the energy of particles. In \cite{rein1994static}, Rein extended the above result for distribution functions depending also on $\ell$, where $\Phi$ is similar to the polytropic ansatz \eqref{polytropes}: $\displaystyle \Phi(E, \ell) = \phi(E)(\ell -\ell_0)_+^k$, $\phi(E) = 0$ if $E>E_0$, $k>-1$ and $\ell_0\geq 0$. Among these, there are singularity-free solutions with a regular center, and also solutions with a Schwarzschild-like black hole.  Other results beyond spherical symmetry were recently established. In \cite{andreasson2011existence},  static and axisymmetric solutions to EV system were constructed by Rein-Andréasson-Kunze and then extended in \cite{andreasson2014rotating} to establish the existence of rotating stationary and axisymmetric solutions to the Einstein-Vlasov system. The  constructed solutions are obtained as bifurcations of a spherically symmetric Newtonian steady state. We note that the proof is based on the application of the implicit function theorem and it can be traced back to Lichtenstein \cite{lichtenstein2013gleichgewichtsfiguren} who proved the existence of non-relativistic, stationary, axisymmetric self-gravitating fluid balls.  Similar arguments were adapted by Andréasson-Fajman-Thaller in \cite{andreasson2015static}  to prove the existence of static spherically symmetric solutions of the Einstein-Vlasov-Maxwell system with non-vanishing  cosmological constant. Recently, Thaller has constructed in \cite{thaller2019existence} static and axially symmetric shells of the Einstein-Vlasov-Maxwell system.
\\
 \\In this paper, we construct static spherically symmetric solutions to the EV system which contain a matter shell located in the exterior region of the Schwarzschild black hole. This provides an alternative construction to that of  Rein \cite{rein1994static}, See remarks \ref{rem1} and \ref{rem2} below for a comparison of our works and results. Our proof is based on the study of  trapped timelike geodesics of spacetimes closed to Schwarzschild. In particular, we show (and exploit) that for some values of energy and total angular momentum $(E, \ell)$, the effective potential associated to a particle moving in a perturbed Schwarzschild spacetime and that of a particle with same $(E, \ell)$ moving in  Schwarzschild are similar. Our distribution function will then be supported on the set of trapped timelike geodesics, and this will lead to the finiteness of the mass and the matter shell's radii. 
 \\
 \\Recall that the effective potential energy $E_\ell^{Sch}$ of a particle of rest mass $m=1$ and angular momentum $\ell$, moving in the exterior of region of the Schwarzschild spacetime is given by 
\begin{figure}[h!]
\includegraphics[width=\linewidth]{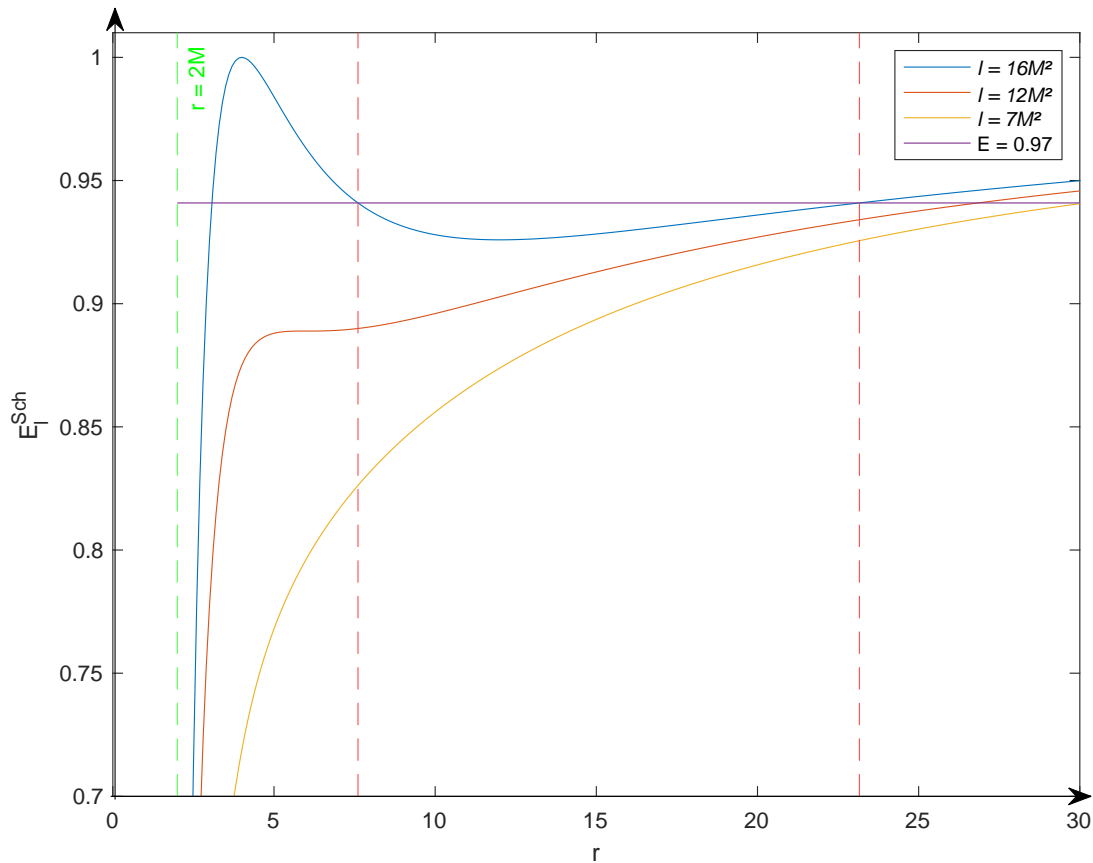}
\label{Fig::2}
\caption{Shape of the effective potential energy $E_\ell^{Sch}$ for three cases of $\ell$, with $M=1$, $E = 0.97$. }
\end{figure}
 \begin{equation}
 E_\ell^{Sch}(r) = \left(1 - \frac{2M}{r} \right)\left(1 + \frac{\ell}{r^2} \right). 
 \end{equation}
 In particular, trapped geodesics occur when the equation 
 \begin{equation}
 \label{motion-}
 E_\ell^{Sch}(r) = E^2,
 \end{equation}
 admits three distinct roots $r_0^{Sch}(E, \ell) < r_1^{Sch}(E, \ell) < r_2^{Sch}(E, \ell)$.  
 \\We state now our main result:   
\begin{theoreme}
\label{thm::1}
There exists a $1-$parameter family of smooth, static, spherically symmetric asymptotically flat spacetimes $(\spacetime, g_\delta)$ and distribution functions $f^\delta: \Gamma_1\to \mathbb R_+$ solving the Einstein-Vlasov system given by equations \eqref{EFE}, \eqref{Vlasov2} and \eqref{EM_tensor}, such that $f^\delta$ verifies 
\begin{equation}
\forall(x, v)\in\Gamma_1,\; f^\delta(x,v) = \Phi(E^\delta, \ell; \delta)\Psi(r, (E^\delta, \ell), g_\delta). 
\end{equation}
where $\Phi(\cdot, \cdot;\delta)$ is supported on a compact set $B_{bound}$ of the set of parameters $(E, \ell)$ corresponding to trapped timelike trajectories, $\Psi$ is a positive cut-off function, defined below (cf. \eqref{cut::off}) which selects the trapped geodesics with parameters $(E, \ell)\in B_{bound}$, $\Gamma_1$ is the mass shell of particles with rest mass $m=1$, and $E^\delta$ is the local energy with respect to the metric $g_\delta$. 
\\Moreover, the resulted spacetimes contain a shell of Vlasov matter in the following sense: there exist $\displaystyle R_{min} < R_{max}$ such that the metric is given by Schwarzschild metric of mass $M$ in the region $]2M, R_{min}]$, by Schwarzschild metric of mass $M^\delta$ in the region  $[R_{max}, \infty[ $, and  $f^\delta$ does not identically vanish in the region $[R_{min}, R_{max}]$.  
\end{theoreme}
\begin{remark}
\label{rem1}
Compared to the work of Rein \cite{rein1994static}, the solutions constructed here are bifurcating from Schwarzschild and therefore, a priori small.  Moreover, the ansatz used in this work is different from that used by Rein. In particular, it is not polytropic. Another difference  is that we do not use the Tolman-Oppenheimer-Volkov equation \eqref{TOV::equation} in our argument. See Theorem \ref{main::result} for the exact assumptions on the profile $\Phi$.
\end{remark}
\begin{remark}
\label{rem2}
A posteriori, in the small data case, one can check that the distribution functions constructed by Rein are supported on trapped timelike geodesics. However, not all the trapped geodesics are admissible in his construction. In our case, we include a more general support, cf. Section \ref{Rein:work}  for more details on this part. 
\end{remark}
\begin{remark}
The support of $\Phi(E, \ell; \delta)$ has two connected components: one corresponds to geodesics which start  near $r_0(E, \ell)$ and reach the horizon in a finite proper time, and the other one corresponds to trapped geodesics. $\Psi$ is introduced so that it is equal to $0$ outside $B_{bound}$ and equal to a cut-off function depending on the $r$ variable (cf \eqref{cut:off:bis}),  $\Chi$ on $B_{bound}$. The latter is equal to $0$ on the first connected component of the support of $\Phi(E, \ell; \delta)$ and to $1$ on the second component. This allows to eliminate the undesired trajectories.  The reason behind the use of this cut-off function is related to the non-validity of Jean's theorem in general relativity \cite{schaeffer1999class}.  
\end{remark}

The paper is organised as follows.  In  Section \ref{Preliminaries}, we present basic background material on the Einstein-Vlasov system and the geodesic motion in the Schwarzschild exterior. We also present the ansatz for the metric and the distribution function and we reduce the Einstein equations to a system of ordinary differential equations in the metric components. We end this section with a short description of Rein's construction to compare it with ours. In Section \ref{main:result},  we give a detailed formulation of  our main result. In Section \ref{Set-up}, we prove Theorem \ref{thm::1}. To this end, we control quantitatively the effective potential and the resulting trapped timelike geodesics for static spherically symmetric spacetimes close to Schwarzschild. The main theorem is then obtained by application of the implicit function theorem. Finally, Appendix \ref{Schwarzschild} contains a proof of Proposition \ref{Propo1} concerning the classification of timelike geodesics in the Schwarzschild spacetime.

\begin{thank}
I would like to thank my advisor Jacques Smulevici for suggesting this problem to me, as well for many interesting discussions and crucial suggestions. I would also like to thank Nicolas Clozeau for  many helpful remarks. This work was supported by the ERC grant 714408 GEOWAKI, under the European Union’s Horizon 2020 research and innovation program.
\end{thank}

\section{Preliminaries and basic background material}
\label{Preliminaries}
In this section, we introduce basic material necessary for the rest of the paper.
\subsection{The Einstein-Vlasov system}
In this work, we study the Einstein field equations for a time oriented Lorentzian manifold  $(\spacetime, g)$ in the presence of matter
\begin{equation}
\label{EFE}
    Ric(g)-\frac{1}{2}gR(g) = 8\pi T(g),
\end{equation}
where $Ric$ denotes the \textit{Ricci curvature tensor} of $g$, $R$ denotes the \textit{scalar curvature} and $T$ denotes the \textit{energy-momentum tensor}  which must be specified by the matter theory. The model considered here is  the Vlasov matter. It is assumed that the latter  is represented by a scalar positive function $f:T\spacetime\to \mathbb R$  called the \textit{distribution function}. The condition that $f$ represents the distribution of a collection of particles moving freely in the given spacetime is that it should be constant along the \textit{geodesic  flow}, that is
\begin{equation}
\label{Liouville::}
L[f] = 0, 
\end{equation} 
where $L$ denotes the \textit{Liouville vector field}. The latter equation is called \textit{the Vlasov equation}.  In a local coordinate chart $(x^\alpha, v^\beta)$ on $T\spacetime$, where $(v^\beta)$ are the components of the four-momentum corresponding to $x^\alpha$, the Liouville vector field  $L$ reads  
\begin{equation}
\label{Liouville:VF}
L = v^\mu\frac{\partial }{\partial x^\mu}-\Chris{\mu}{\alpha}{\beta}(g)v^\alpha v^\beta\frac{\partial }{\partial v^\mu}
\end{equation}
and the corresponding integral curves satisfy the geodesic equations of motion
\begin{equation}
\label{eq::motion1}
\left\{
\begin{aligned}
&\frac{dx^\mu}{d\tau}(\tau) = v^\mu, \\
&\frac{dv^\mu}{d\tau}(\tau) = -\Chris{\mu}{\alpha}{\beta}v^\alpha v^\beta,
\end{aligned}
\right. 
\end{equation}
where $\displaystyle \Chris{\mu}{\alpha}{\beta}$ are the Christoffel symbols given in the local coordinates $x^\alpha$ by 
\begin{equation*}
\Chris{\mu}{\alpha}{\beta} = \frac{1}{2}\ginv{\mu}{\nu}\left( \frac{\partial\g{\beta}{\nu}}{\partial x^\alpha}+\frac{\partial\g{\alpha}{\nu}}{\partial x^\beta} - \frac{\partial \g{\alpha}{\beta}}{\partial x^\nu}   \right)
\end{equation*}
and where $\tau$ is an affine parameter which corresponds to the proper time in the case of timelike geodesics. The trajectory of a particle in $T\spacetime$ is an element of the geodesic flow generated by $L$ and its projection onto the spacetime manifold $\spacetime$ corresponds to a geodesic of the spacetime. 
\\It is easy to see that the quantity $\displaystyle \mathcal L(x,v) := \frac{1}{2}v^\alpha v^\beta\g{\alpha}{\beta}$ is conserved along solutions of \eqref{eq::motion1} \footnote{We note that $\mathcal L$ is the Lagrangian of a free-particle.}. In the case of timelike geodesics this corresponds to the conservation of the rest mass $m>0$ of the particle. In the following, we will consider only particles with the same rest mass so that we can set $\displaystyle \mathcal L(x,v) = -\frac{m^2}{2}$. Furthermore, for physical reasons, we require that all particles move on future directed timelike geodesics. Therefore, the distribution function is supported on the seven dimensional manifold of $T\spacetime$ \footnote{See \cite{sarbach2014geometry}  Lemma.7}, called the \textit{the mass shell}, denoted by $\Gamma_m$ and defined by 
\begin{equation}
\label{mass::shell}
\Gamma_m := \left\{ (x,v)\in T\spacetime : g_x(v,v)= -m^2,  \quad\text{and}\quad v^\alpha  \text{ is future pointing}\right\}. 
\end{equation}
We note that by construction $\Gamma_m$ is invariant under the geodesic flow.  
\\ 
\\  We assume that there exist local coordinates on $\spacetime$, denoted by $(x^\alpha)_{\alpha=0\cdots 3}$ such that $x^0 = t$ is a smooth strictly increasing function along any future causal curve and its gradient is past directed and timelike. Then, the condition 
\begin{equation*}
\g{\alpha}{\beta}v^\alpha v^\beta = -1 \quad\quad \text{where}\; v^\alpha \text{is future directed }
\end{equation*}
allows to write $v^0$ in terms of $(x^\alpha, v^a)$. It is given by 
\begin{equation*}
v^0 = -(\g{0}{0})^{-1}\left(\g{0}{j}v^j +\sqrt{(\g{0}{j}v^j)^2-\g{0}{0}(m^2 + \g{i}{j}v^iv^j)}\right).
\end{equation*}
Therefore, $\Gamma_m$ can be parametrised by $(x^0, x^a, v^a)$.  Hence, the distribution function can be written as a function of $(x^0, x^a, v^a)$ and the Vlasov equation has the form
\begin{equation}
\label{Vlasov2}
\displaystyle \frac{\partial f}{\partial x^0} + \frac{v^a}{v^0}\frac{\partial f}{\partial x^a} - \Chris{a}{\alpha}{\beta}\frac{v^\alpha v^\beta}{v^0}\frac{\partial f}{\partial v^a} = 0.
\end{equation}
In order to define the energy-momentum tensor which couples the Vlasov equation to the Einstein field equations, we introduce the natural  volume element on the fibre $\Gamma_{m,x}:= \left\{ v^\alpha\in T_x\spacetime \;:\;  \ginv{\alpha}{\beta}v^\alpha v^\beta = -m^2, \; v^0>0 \right\}$ of $\Gamma_m$ at a point $x\in\spacetime$ given in  the adapted local coordinates  $(x^0, x^a, v^a)$ by
\begin{equation}
\label{vol::form}
d\text{vol}_x(v) := \frac{\sqrt{-\det{(\g{\alpha}{\beta})}}}{-v_0}dv^1dv^2dv^3.  
\end{equation}
The energy momentum tensor is now defined by
\begin{equation}
\label{EM_tensor}
\forall x\in\spacetime\;\quad\T{\alpha}{\beta}(x):= \int_{\Gamma_{m,x}} v_\alpha v_\beta f(x,v)\;d\text{vol}_x(v),
\end{equation}
where $f = f(x^0,x^a, v^a)$ and $d\text{vol}_x(v) = d\text{vol}_x(v^a)$\footnote{The latin indices run from 1...3.}. In order for \eqref{EM_tensor} to be well defined, $f$ has to have certain regularity and decay properties. One sufficient requirement would be to demand that $f$ have compact support on  $\Gamma_x$, $\forall x\in \spacetime$ and it is integrable with respect to $v$. 
In this work, we assume that all particles have the same rest mass and it is normalised to $1$.  From now on, the mass shell $\Gamma_1$ will be denoted by $\Gamma$. Finally, we refer to \eqref{EFE} and \eqref{Vlasov2} with $T$ given by \eqref{EM_tensor} as the Einstein-Vlasov system.

\subsection{Static and Spherically symmetric solutions}

\subsubsection{Metric ansatz}
We are looking for static and spherically symmetric asymptotically flat solutions to the EV system. Therefore, we consider the following ansatz for the metric, written in standard $\displaystyle (t, r, \theta, \phi)$ coordinates:
\begin{equation}
\label{metric::ansatz}
g= -e^{2\mu(r)} dt^2 + e^{2\lambda(r)}dr^2 + r^2(d\theta^2 + \sin^2\theta d\phi^2).
\end{equation}
with boundary conditions at infinity: 
\begin{equation}
\lim_{r\to\infty} \lambda(r) = \lim_{r\to\infty} \mu(r) = 0. 
\end{equation}
Since we consider solutions close to a Schwarzschild spacetime of mass $M>0$ outside from the black hole region, we fix a manifold of the form
\begin{align}
\label{domain::}
\mathcal O := &\mathbb R\times]2M, \infty[\times S^2.  \\
& \quad t\quad\quad r \quad \quad\theta, \phi 
\end{align} 

\subsubsection{Vlasov field on static and spherically symmetric spacetimes}
The distribution function $f$ is conserved along the geodesic flow. Hence, any function of the integrals of motion will satisfy the Vlasov equation. In this context, we  look for integrals of motion for the geodesic equation \eqref{eq::motion1} on a static and spherically symmetric background. By symmetry assumptions, the vector fields 
 \begin{equation}
\xi_t = \frac{\partial}{\partial t},
\end{equation}
generating stationarity, and 
\begin{equation}
\xi_1 = -\cos\phi\cot\theta\frac{\partial}{\partial\phi}-\sin\phi\frac{\partial}{\partial\theta},
\end{equation}
\begin{equation}
\xi_2 = -\sin\phi\cot\theta\frac{\partial}{\partial\phi}+\cos\phi\frac{\partial}{\partial\theta},
\end{equation}
\begin{equation}
\xi_3 = \frac{\partial}{\partial \phi}
\end{equation}
generating spherical symmetry are Killing. Therefore, the quantities 
\begin{equation}
E(x,v) := -g_x(v,\xi_t) = e^{\mu(r)}\sqrt{m^2+ (e^{\lambda(r)}v^r)^2 + (rv^\theta)^2 + (r\sin\theta v^\phi)^2} ,
\end{equation}
\begin{equation}
\ell_1(x,v) := g_x(v,\xi_1) = r^2\left(-\sin\phi v^\theta - \cos\theta\sin\theta\cos\phi v^\phi\right),
\end{equation}
\begin{equation}
\ell_2(x,v) := g_x(v,\xi_2) = r^2\left(\cos\phi v^\theta - \cos\theta\sin\theta\sin\phi v^\phi\right),
\end{equation}
\begin{equation}
\ell_z(x,v) := g_x(v,\xi_3) = r^2\sin^2\theta v^\phi,
\end{equation}
are conserved. In particular, $E$ and
\begin{equation}
\label{ang::momentum}
\ell:= \ell(r, \theta, v^\theta, v^\phi) = \ell_1^2+\ell_2^2+\ell_z^2 = r^4\left((v^\theta)^2 + \sin^2\theta(v^\phi)^2\right)
\end{equation}
are conserved. We recall that $m:= -\g{\alpha}{\beta}v^\alpha v^\beta$ is also a conserved quantity. Since all particles have the same rest mass $m=1$, $E$ becomes
\begin{equation}
\label{energy}
E = E(r, v^r, v^\theta, v^\phi) = e^{\mu(r)}\sqrt{1+ (e^{\lambda(r)}v^r)^2 + (rv^\theta)^2 + (r\sin\theta v^\phi)^2}. 
\end{equation}
Note that $E$ and $\ell$ can be interpreted respectively as the energy and the total angular momentum of the particles. 

\subsection{The Schwarzschild spacetime and its timelike geodesics}
The Schwarzschild family of spacetimes is a one-parameter family of spherically symmetric  Lorentzian manifolds, indexed by $M>0$, which are solutions to the Einstein vacuum equations
\begin{equation}
\label{EVE}
Ric(g) = 0.
\end{equation}  
The parameter $M$ denotes the ADM mass. The \textit{domain of outer communications} of these solutions can be represented by $\mathcal O $ and a metric which takes the form 
\begin{equation}
\label{metric:ansatz}
g_{Sch}= -e^{2\mu^{Sch}(r)} dt^2 + e^{2\lambda^{Sch}(r)}dr^2 + r^2(d\theta^2 + \sin^2\theta d\phi^2),
\end{equation}
where 
\begin{equation}
\label{Sch::met}
e^{2\mu^{Sch}(r)} = 1 - \frac{2M}{r}\quad\text{and}\quad e^{2\lambda^{Sch}(r)} = \left(1 - \frac{2M}{r}\right)^{-1}.
\end{equation}
%, bounded by the the two surface $\displaystyle \mathcal H := \left\{ (t,r,\theta,\phi) \;:\; t\in\mathbb R, \quad r = 2M, \quad \theta\in(0,\pi), \quad \phi\in0,2\pi)\right\}$. In particular we are interested in the ranges for $(E, \ell)$ of a particle moving along trapped geodesic.
%We note there are no coordinate singularities for $r>2M$ except at the poles $\theta = 0, \pi$. The latter ones can be removed by working in Cartesian coordinates $(x^1, x^2, x^3) = r(\cos\phi\sin\theta, \sin\phi\sin\theta, \cos\theta)$ instead of $(r, \theta, \phi)$.  
In this work, we shall be interested in particles moving in the exterior region. We note that the study of the geodesic motion is included in the classical books of general relativity. See for example \cite[Chapter 3]{Chandrasekhar:1985kt} or \cite[Chapter 33]{misner2017gravitation}. We recall here a complete classification of timelike geodesics in order to be self contained. 
\\The delineation of Schwarzschild's geodesics requires solving the geodesic equations given by \eqref{eq::motion1}. Therefore, we need to integrate a system of 8 ordinary differential equations. However, the symmetries of the Schwarzschild spacetime and the mass shell condition 
\begin{equation}
\label{m:s:condition}
\mathcal L(x,v) = -\frac{m^2}{2}
\end{equation}
 imply the complete integrability of the system.
\\
\\Let $(x^\alpha, v^\alpha) = (t, r, \theta, \phi, v^t, v^r, v^\theta, v^\phi)$ be a local coordinate system on $T\spacetime$. The quantities 
\begin{equation*}
E = -v_t, \quad \ell = r^4\left( (v^\theta)^2 + \sin^2\theta(v^\phi)^2\right), \quad \ell_z = r\sin\theta v^\phi.
\end{equation*}
are then conserved along the geodesic flow. Besides, the mass shell condition on $\Gamma$ takes the form
\begin{equation*}
-e^{2\mu^{Sch}(r)}(v^t)^2 + e^{2\lambda^{Sch}(r)}(v^r)^2 + r^2\left( (v^\theta)^2 + \sin^2\theta(v^\phi)^2\right) = -1
\end{equation*}
which is equivalent to 
\begin{equation}
\label{r::pr}
E^2 = \Esch(r) + (e^{\lambda^{Sch}}v^r)^2, \quad\text{where}\quad \Esch(r) = e^{2\mu^{Sch}(r)}\left( 1+ \frac{\ell}{r^2}\right). 
\end{equation}
and implies
\begin{equation}
\label{pot::ineq}
\Esch(r) \leq E^2
\end{equation}
for any geodesic moving in the exterior region. 
\\Let $\gamma : I \to\spacetime$ be a timelike geodesic in the spacetime, defined on interval $ I \subset\mathbb R$. In the adapted local coordinates $(x^\alpha, v^\alpha)$, we have 
\begin{equation*}
\gamma(\tau) = (t(\tau), r(\tau), \theta(\tau), \phi(\tau)) \quad\text{and}\quad \gamma'(\tau) = (v^t(\tau), v^r(\tau), v^\theta(\tau), v^\phi(\tau)) .
\end{equation*}
Besides, $\gamma$ satisfies the geodesic equations of motion  \eqref{eq::motion1}. One can easily see from the equations 
\begin{align}
\label{eq1}
v_t = -E, \\
\label{eq2}
\quad   r\sin\theta v^\phi = \ell_z, \\
\label{eq3}
\quad  (rv^\theta)^2+\frac{\ell_z^2}{\sin^2\theta} = \ell, \\
\label{eq4}
\quad (e^{\lambda^{Sch}(r)}v^r)^2  - e^{-2\mu^{Sch}(r)} E^2 + \left( 1 + \frac{\ell}{r^2}\right) = 0
\end{align}
that if we solve the geodesic motion in the radial direction $r(\tau)$ then we can integrate the remaining equations of motion. More precisely, if one solves for $r$ then we get $t$ from \eqref{eq1}, $\phi$ from \eqref{eq2} and $\theta$ from \eqref{eq3}. Therefore, we will study the geodesic equation projected only in the radial direction i.e we consider the reduced system
\begin{align}
\frac{dr}{d\tau} &= v^r ,\\
\frac{dv^r}{d\tau} &= -\Chris{r}{\alpha}{\beta}v^\alpha v^\beta.
\end{align} 
Straightforward computations of the right hand side of the second equation lead to 
\begin{equation*}
\Chris{r}{\alpha}{\beta}v^\alpha v^\beta = \frac{M}{r^2}\left(1-\frac{2M}{r}\right)^{-1}E^2 - \frac{M}{r^2}\left(1-\frac{2M}{r}\right)^{-1} (v^r)^2 - \frac{1}{r^3}\left(1-\frac{2M}{r}\right)\ell. 
\end{equation*}
We combine the latter expression with \eqref{eq4}, to obtain
\begin{equation*}
\frac{dv^r}{d\tau} = -\frac{M}{r^2}\left(1+\frac{\ell}{r^2}\right) + \frac{\ell}{r^3}\left(1 - \frac{2M}{r}\right).
\end{equation*}
Now, by \eqref{r::pr}, it is easy to see that 
\begin{equation*}
\frac{dv^r}{d\tau} = -\frac{1}{2}{\Esch}'(r). 
\end{equation*}
In order to lighten the notations, we denote $v^r$ by $w$. Thus, we consider the differential system 
\begin{align}
\label{r:motion}
\frac{dr}{d\tau} &= w ,\\
\label{pr:motion}
\frac{dw}{d\tau} &= -\frac{1}{2}{\Esch}'(r).
\end{align} 
We classify solutions of \eqref{r:motion}-\eqref{pr:motion} based on the shape of $\Esch$ as well as the roots of the equation 
\begin{equation}
\label{eq::motion:}
E^2 = \Esch(r). 
\end{equation}
More precisely, we state the following proposition. We include a complete proof of this classical result in Appendix \ref{Schwarzschild}. 
\begin{Propo}
\label{Propo1}
Consider a timelike geodesic $\gamma:I\to\spacetime$ of the Schwarzschild exterior of mass $M>0$  parametrised by $\displaystyle \gamma(\tau) = (t(\tau), r(\tau), \theta(\tau), \phi(\tau))$ and normalised so that $\displaystyle g_{Sch}(\dot\gamma, \dot\gamma) = -1$. Let $E$ and $\ell$ be the associated energy and total angular momentum defined respectively by $\displaystyle  E := -v_t $ and  $ \ell := r^4\left( (v^\theta)^2 + \sin^2\theta(v^\phi)^2\right)$. Then, we have the following classification
\begin{enumerate}

\item If $\displaystyle \ell\leq 12M^2$, then

\begin{enumerate}
\item if $\displaystyle \sqrt{\frac{8}{9}}<E< 1$, then the orbit starts at some point $r_0>2M$, the unique root of the equation \eqref{eq::motion:}, and reaches the horizon $r=2M$ in a finite proper time. 
\item if $E\geq 1$, then the equation \eqref{eq::motion:} admits no positive roots, and the orbit goes to infinity $r = +\infty$ in the future, while in the past, it reaches the horizon $r=2M$ in finite affine time. 
\item if $\displaystyle E = \sqrt{\frac{8}{9}}$, then $\ell = 12M^2$. In this case,  the equation \eqref{eq::motion:} admits a unique triple root, given by $r_c = 6M$. The orbit is circular of radius $r_c$.
\end{enumerate}

\item If $\displaystyle \ell> 12M^2$, then

\begin{enumerate}

\item if $\ell = \ell_{lb}(E)$, the geodesic is a circle of radius $r_{max}^{Sch}(\ell)$, given by

\begin{equation}
\label{r::max}
r_{max}^{Sch}(\ell) = \frac{\ell}{2M}\left( 1 - \sqrt{1-\frac{12M^2}{\ell}}\right).
\end{equation} 
Here, $\ell_{lb}(E)$ is given by
\begin{equation}
\label{l::lb}
\ell_{lb}(E) := \frac{12M^2}{1-4\alpha-8\alpha^2+8\alpha\sqrt{\alpha^2+\alpha}}, \quad\quad \alpha := \frac{9}{8}E^2 - 1.
\end{equation}

\item if $\ell < \ell_{lb}(E)$,  

\begin{enumerate}
\item if $E\geq 1$, then the equation \eqref{eq::motion:} admits no positive roots. Hence, the orbit goes to infinity $r = +\infty$ in the future, while in the past, it reaches the horizon $r=2M$ in finite affine time. 
\item otherwise, the orbit starts at some point $r_0> 2M$, the unique root of the equation \eqref{eq::motion:}, and reaches the horizon ($r=2M$) in a finite proper time. 
\end{enumerate}

\item otherwise, if $\ell > \ell_{lb}(E)$, 

\begin{enumerate}

\item if $E\geq 1$, then the equation \eqref{eq::motion:} admits two positive distinct roots $2M<r_0 < r_2$ and two orbits are possible 

\begin{enumerate}
\item The geodesic starts at $r_0$ and reaches the horizon in a finite proper time. 
\item The orbit is hyperbolic, so unbounded. The orbit starts from infinity, hits the potential barrier at $r_2$, the biggest root of the equation \eqref{eq::motion:},  and goes back to infinity.  
\end{enumerate} 

\item otherwise, if $E< 1$,  and necessarily $\displaystyle \ell \leq \ell_{ub}(E)$, where $\ell_{ub}(E)$ is given by

\begin{equation}
\label{l::ub}
\ell_{ub}(E) := \frac{12M^2}{1-4\alpha-8\alpha^2-8\alpha\sqrt{\alpha^2+\alpha}},   
\end{equation}

\begin{enumerate}

\item if $\ell = \ell_{ub}(E)$, the geodesic is a circle of radius $r_{min}^{Sch}(\ell)$, given by 

\begin{equation}
\label{r::min}
r_{min}^{Sch}(\ell) = \frac{\ell}{2M}\left( 1 + \sqrt{1-\frac{12M^2}{\ell}}\right). 
\end{equation} 

\item If $\ell < \ell_{ub}(E)$, then the equation \eqref{eq::motion:} admits three distinct roots $2M<r_0^{Sch}(E, \ell)< r_1^{Sch}(E, \ell)< r_2^{Sch}(E, \ell)$, and two orbits are possible

\begin{itemize}
\item The geodesic starts from $r_0^{Sch}(E, \ell)$ and reaches the horizon in a finite time.
\item The geodesic is periodic. It oscillates between an aphelion $r_2^{Sch}(E, \ell)$ and a perihelion $r_1^{Sch}(E, \ell)$.  
\end{itemize}

%\end{enumerate}

\end{enumerate}

\end{enumerate}

\end{enumerate} 

\end{enumerate}

\end{Propo}
\begin{remark}
In view of \eqref{pot::ineq}, the lower bound $\displaystyle E\geq \sqrt{\frac{8}{9}}$ holds along any orbit.  
\end{remark}
In view of the above proposition, we introduce $\Abound$, the set of  $(E,\ell)$ parameters corresponding to trapped (non-circular) geodesics. 
\begin{definition}
\begin{equation}
\label{A::bound}
\mathcal  A_{bound} := \left\{ (E, \ell)\in\ELrange \;:\; E <1 ,\quad \ell_{lb}(E) <\ell< \ell_{ub}(E) \right\}. 
\end{equation}
\end{definition}
Note in view of the orbits described in $2.c.ii.B$ that given $(E, \ell)\in\Abound$, two cases are possible: the orbit is either trapped or it reaches the horizon in a finite proper time.  
\subsection{Ansatz for the distribution function}
We are interested in static and spherically symmetric distribution functions. For simplicity, we fix the rest mass of the particles to be $1$. In this context, we assume that $f:\Gamma_1\to\mathbb R_+$ takes the form
\begin{equation}
\label{f::ansatz}
f(t,r,\theta,\phi, v^r, v^\theta, v^\phi) =  \Phi(E, \ell)\Psi_\eta(r, (E, \ell), \mu).
\end{equation}
where 
\begin{itemize}
\item $\displaystyle \Phi : \Abound \to \mathbb R_+$ is a $C^2$ function on its domain and it is supported on $\displaystyle B_{bound}$, a set of the form $\displaystyle  [E_1, E_2]\times[\ell_1, \ell_2]$, where $E_1, E_2$, $\ell_1$ and $\ell_2$ verify
\begin{equation}
\label{bounds:}
\sqrt{\frac{8}{9}}<E_1<E_2<1 \quad\text{and}\quad \ell_{lb}(E_2)<\ell_1<\ell_2<\ell_{ub}(E_1).
\end{equation} 
\item $\eta> 0$ is constant that will be specified later, $\Psi_\eta(\cdot, \cdot, \mu)\in C^\infty(\mathbb R\times\Abound, \mathbb R_+)$ is a cut-off function depending on the metric coefficient $\mu$, such that     
\begin{equation}
\label{cut::off}
\Psi_\eta(\cdot, (E, \ell), \mu) := 
\left\{
\begin{aligned}
& \Chi_\eta(\cdot - r_1(\mu, (E, \ell))), \quad (E, \ell)\in B_{bound}\\
& 0 \quad\quad\quad\quad\quad\quad\quad\quad\quad\quad (E, \ell)\notin \Abound,
\end{aligned}
\right.
\end{equation}
where $r_1$ is a positive function of $(\mu, E, \ell)$ which will be defined later \footnote{$r_1$ will be the second largest root of the equation $\displaystyle e^{2\mu(r)}\left(1 + \frac{\ell}{r^2} \right) = E^2$ corresponding to the metric $g$.} and $\Chi_\eta\in C^\infty(\mathbb R, \mathbb R_+)$ is a cut-off function defined by 

\begin{equation}
\label{cut:off:bis}
\Chi_\eta(s) :=
\left\{
\begin{aligned}
& 1 \quad s\geq 0 , \\
& \leq1 \quad s\in [-\eta, 0] , \\
& 0 \quad s< -\eta. 
\end{aligned}
\right. 
\end{equation}

\item $E$ and $\ell$ are defined respectively by \eqref{energy}, \eqref{ang::momentum}.
\end{itemize}
Based on the monotonicity properties of $\ell_{ub}$, $\ell_{ub}$ $r_{min}^{Sch}$ and $r_{max}^{Sch}$, the set $B_{bound}$ is included in $\Abound$. More precisely, we state the following lemma 
\begin{lemma}
\label{monotonicity:1}
Let $\displaystyle\ell\in]12M^2, \infty[$
\begin{itemize}
\item $\displaystyle r_{max}^{Sch}$ decreases monotonically from $6M$ to $3M$ on $\displaystyle ]12M^2, \infty[$.
\item $\displaystyle r_{min}^{Sch}$ increases monotonically from $6M$ to $\infty$ on $\displaystyle ]12M^2, \infty[$. 
\end{itemize}
Let $\displaystyle E\in\left]\sqrt{\frac{8}{9}}, 1\right[$,
\begin{itemize}
\item $\displaystyle \ell_{lb}$ grows monotonically from $12M^2$ to $16M^2$ when $E$ grows from $\displaystyle \sqrt{\frac{8}{9}}$ to $1$.
\item $\displaystyle \ell_{ub}$ grows monotonically from $12M^2$ to $\infty$ when $E$ grows from $\displaystyle \sqrt{\frac{8}{9}}$ to $1$.
\end{itemize}
Moreover, $\forall (E, \ell)\in[E_1, E_2]\times\left]\ell_{lb}(E_2), \ell_{lb}(E_1) \right[$
 \begin{equation*}
 \ell_{lb}(E)< \ell_{lb}(E_2)< \ell<\ell_{ub}(E_1) < \ell_{ub}(E).  
\end{equation*}
\end{lemma}
\begin{proof}
The proof is straightforward in view of \eqref{r::max} , \eqref{r::min}, \eqref{l::lb} and \eqref{l::ub}.
\end{proof}
Hence,  by the above lemma, $\displaystyle [E_1, E_2]\times\left]\ell_{lb}(E_2), \ell_{lb}(E_1) \right[\subset\Abound$ and in particular, $B_{bound}\subset\subset \Abound$.
\\ Finally, we note that all trapped geodesics lie in the region $\left]4M, \infty\right[$:
 \begin{lemma}
\label{monotonicity}
Let $\displaystyle (E, \ell)\in B_{bound}$. Let $\displaystyle r_i^{Sch}(E, \ell)\footnote{We note that the dependence of $r_i$ in $(E, \ell)$ is smooth.}, \; i\in\left\{0, 1, 2\right\}$ be the three roots of the equation 
\begin{equation}
\label{eq:lemma2}
e^{2\mu^{Sch}(r)}\left(1 + \frac{\ell}{r^2} \right) = E^2. 
\end{equation}
Then, 
\begin{itemize}
\item $\displaystyle r_0^{Sch}(E, \cdot)$ decreases monotonically on $\left]\ell_1, \ell_2\right[$ and  $\displaystyle r_0^{Sch}(\cdot, \ell)$ increases monotonically  on $\left] E_1, E_2 \right[$. 
\item $\displaystyle r_1^{Sch}(E, \cdot)$ increases monotonically on $\left]\ell_1, \ell_2\right[$ and $\displaystyle r_1^{Sch}(\cdot, \ell)$ decreases monotonically  on $\left] E_1, E_2 \right[$. 
\item $\displaystyle r_2^{Sch}(E, \cdot)$ decreases monotonically on $\left]\ell_1, \ell_2\right[$ and $\displaystyle r_2^{Sch}(\cdot, \ell)$ increases monotonically  on $\left] E_1, E_2 \right[$.
 \end{itemize}
 Moreover, 
 \begin{equation*}
 \forall(E, \ell)\in A_{bound}\quad:\quad r^{Sch}_1(E, \ell) > 4M.   
 \end{equation*}
\end{lemma}
\begin{proof}
First we note that the equation \eqref{eq:lemma2} is cubic in $\displaystyle \frac{1}{r}$.  Its roots can be obtained analytically by Cardan's formula and in particular, depend smoothly on $(E, \ell)$. Besides, the roots are simple. Otherwise, the orbit would be circular. Therefore,
\begin{equation*}
(\Esch)'(r_i^{Sch}(E, \ell)) \neq 0. 
\end{equation*}   
$r^{Sch}_i$ verify
\begin{equation*}
E^2 = \Esch(r_i^{Sch}(E, \ell)). 
\end{equation*}
We derive the latter equation with respect to $E$ and $\ell$ to obtain 
\begin{equation*}
\displaystyle
\frac{\partial r_i^{Sch}}{\partial\ell}(E, \ell) = -\frac{\frac{\partial\Esch}{\partial\ell}(r_i^{Sch}(E, \ell))}{(\Esch)'(r_i^{Sch}(E, \ell))}, \quad\text{and}\quad  \frac{\partial r_i^{Sch}}{\partial E}(E, \ell) = -\frac{2E}{(\Esch)'(r_i^{Sch}(E, \ell))}
\end{equation*}
We use the monotonicity properties of $\Esch$ on $]2M, \infty[$ to determine the sign of the above derivatives. For the second point, we use the monotonicity properties of $r_1^{Sch}$, $r_{max}^{Sch}$ and $\ell_{ub}$ to obtain
\begin{equation*}
\forall (E, \ell)\in\Abound\;:\; r_1^{Sch}(E, \ell) > r_1^{Sch}(E, \ell_{lb}(E))  \quad\text{and}\quad r_{max}^{Sch}(\ell_{lb}(E)) >  r_{max}^{Sch}(\ell_{lb}(1)). 
\end{equation*}
Besides, by the definition on $\ell_{ub}$ (the value of the angular momentum such that $E^{max}(\ell) = E^{Sch}_\ell(r^{Sch}_{max}(\ell)) = E^2$), we have 
\begin{equation*}
r_1^{Sch}(E, \ell_{lb}(E)) = r_{max}^{Sch}(\ell_{lb}(E)). 
\end{equation*} 
Therefore, 
\begin{equation*}
r_1^{Sch}(E, \ell)  > r_{max}^{Sch}(\ell_{lb}(1)).
\end{equation*}
To conclude, it suffices to note that  $\ell_{lb}(1) = 16M^2$ so that, when using \eqref{r::max}, we obtain 
\begin{equation*}
r_{max}^{Sch}(\ell_{lb}(1)) = 4M. 
\end{equation*}
This ends the proof. 
\end{proof}
\begin{remark}
The choice of an open set for the range of $\ell$ allow us eliminate the critical values leading to circular orbits which correspond to double roots of the equation \eqref{eq::motion:}:  the points $(E_1, \ell_{ub}(E_1))$ and $(E_2, \ell_{lb}(E_2))$. 
\end{remark}
\subsection{Reduced Einstein equations}
After specifying the ansatz for the metric and for the distribution function, we compute the energy-momentum tensor associated to \eqref{metric::ansatz} and the distribution function \eqref{f::ansatz}. In order to handle the mass shell condition, we introduce the new momentum variables
\begin{equation}
\label{new::momentum}
p^0:= -e^{\mu(r)}v^t,\quad p^1 := e^{\lambda(r)}v^r,\quad p^2 := rv^\theta,\quad p^3 := r\sin\theta v^\phi,
\end{equation}
so that the associated frame 
\begin{equation}
e_0 := e^{-\mu(r)} \frac{\partial }{\partial t} ,\quad e_1:= e^{-\lambda(r)} \frac{\partial }{\partial r},\quad e_2:= \frac{1}{r} \frac{\partial }{\partial \theta}, \quad e_3:=\frac{1}{r\sin\theta} \frac{\partial }{\partial \phi},
\end{equation}
forms an orthonormal frame of $T_x\spacetime$.
In terms of the new momentum coordinates, the mass shell condition becomes 
\begin{equation*}
-(p^0)^2 + (p^1)^2 + (p^2)^2 + (p^3)^2 = -1 \quad\text{i.e}\quad p^0 = \sqrt{1 + |p|^2}, 
\end{equation*}
where $p = (p^1, p^2, p^3)\in\mathbb R^3$ and $|\cdot|$ is the Euclidean norm on $\mathbb R^3$. It is convient to introduce an angle variable $\chi\in[0, 2\pi[$ such that 
\begin{align*}
v^\theta = \frac{\sqrt{\ell}\cos\chi}{r^2}, \quad\quad  v^\phi = \frac{\sqrt{\ell}\sin\chi}{r^2\sin\theta},
\end{align*}
so that 
\begin{equation*}
r^4(v^\theta)^2 + r^4\sin^2\theta(v^\phi)^2 = \ell.
\end{equation*}
We also introduce the effective potential in the metric \eqref{metric::ansatz}
\begin{equation}
\label{effective::potential}
E_{\ell}(r) := e^{2\mu(r)}\left( 1 + \frac{\ell}{r^2}\right).
\end{equation}
In local coordinates, the energy momentum tensor reads 
\begin{align*}
\T\alpha\beta(x) &= \int_{\Gamma_{x}}\;v_\alpha v_\beta f(x^\mu, v^a)d_{x}\text{vol}(v^a),  \\
&= \int_{\Gamma_{x}}\;v_\alpha v_\beta \Phi(E(x^\mu, v^a), \ell(x^\mu, v^a))\Psi_\eta(r, (E(x^\mu, v^a), \ell(x^\mu, v^a)), \mu)d_{x}\text{vol}(v),
\end{align*}
where  $\displaystyle d_{x}\text{vol}(v)$ is given by \eqref{vol::form}, $E(x^\mu, v^a)$ and $\ell(x^\mu, v^a)$ are respectively given by \eqref{energy} and \eqref{ang::momentum}. Here, $x^\mu = (t, r, \theta, \phi)$ and $v^a = (v^r, v^\theta, v^\phi)$. We make a first change of variables $(v^r, v^\theta, v^\phi)\to (p^1, p^2, p^3)$ to get
\begin{align*}
\T\alpha\beta(x) &= \int_{\mathbb R^3}\;v_\alpha(p) v_\beta(p) \Phi\left(E(r, |p|), \ell(r, p)\right)\Psi_\eta(r, (E(r, |p|), \ell(r, p)), \mu)\;\frac{d^3p}{\sqrt{1+|p|^2}} \\
&= 2\int_{[0, +\infty[\times\mathbb R^2}\;v_\alpha(p) v_\beta(p) \Phi\left(E(r, |p|), \ell(r, p)\right)\Psi_\eta(r, (E(r, |p|), \ell(r, p)), \mu)\;\frac{d^3p}{\sqrt{1+|p|^2}}. 
\end{align*} 
We perform a second change of variables $(p^1, p^2, p^3)\mapsto (E(r, |p|),\ell(r, p), \chi)$. We compute the $p^i$s in terms of the new variables:
\begin{equation*}
(p^1)^2 = E^2(r, |p|)e^{-2\mu(r)} - 1 - \frac{\ell(r, p)}{r^2}, \quad\quad p^2 = \frac{\sqrt{\ell(r, p)}\cos\chi}{r}, \quad\quad p^3 = \frac{\sqrt{\ell(r, p)}\sin\chi}{r}.
\end{equation*}
Now it is easy to see that the domain of $\displaystyle (E, \ell, \chi)$ is given by $\displaystyle D_r\times\left[0,2\pi\right[$ where 
\begin{equation}
\label{D::r}
D_r := \left\{ (E, \ell)\in]0, \infty[\times[0, \infty[\;:\; E_\ell(r) \leq E^2  \right\}. 
\end{equation}
By straightforward computations of the Jacobian, we get
\begin{equation*}
\displaystyle \frac{d^3p}{\sqrt{1+|p|^2}} = \frac{1}{2r^2\sqrt{E^2-E_\ell(r)}}\;dEd\ell d\chi. 
\end{equation*} 
Therefore, the energy momentum tensor becomes 
\begin{equation*}
\T\alpha\beta(x) = \int_{D_r}\int_{[0,2\pi[}u_\alpha u_\beta \Phi(E, \ell)\Psi_\eta(r, (E, \ell, \mu)\frac{1}{r^2\sqrt{E^2 -  E_{\ell}(r)}}\;dE d\ell d\chi,
\end{equation*}
where 
\begin{equation*}
u_t^2 = E^2,\quad u_r^2 = e^{2\lambda(r)-2\mu(r)}\left(E^2 - E_\ell(r)\right), \quad, u_\theta^2 = \ell\cos^2\chi, \quad u_\phi^2 = \ell\sin^2\theta\sin^2\chi. 
\end{equation*}
Hence, the non-vanishing energy-momentum tensor components are given by
\begin{align*}
\displaystyle \T{t}{t}(r) &= \frac{2\pi}{r^2}\int_{D_r}\;E^2\Psi_\eta(r, (E, \ell, \mu)\frac{\Phi(E, \ell)}{\sqrt{E^2 - E_{\ell}(r)}}\;dEd\ell,    \\
\displaystyle  \T{r}{r}(r) &= \frac{2\pi}{r^2}\int_{D_r}\;e^{2(\lambda(r)-\mu(r))}\Psi_\eta(r, (E, \ell, \mu)\Phi(E, \ell)\sqrt{E^2 - E_{\ell}(r)}\;dEd\ell, \\
\displaystyle  \T{\theta}{\theta}(r) &= \frac{\pi}{r^2}\int_{D_r}\; \Psi_\eta(r, (E, \ell, \mu)\Phi(E, \ell)\frac{\ell}{\sqrt{E^2 - E_{\ell}(r)}}\;dEd\ell, \\
\displaystyle  \T{\phi}{\phi}(r, \theta) &= \sin^2\theta\T{\theta}{\theta}(r).
\end{align*}
It remains to compute the Einstein tensor
\begin{equation*}
\Ein{\alpha}{\beta} = \Ric{\alpha}{\beta} - \frac{1}{2}\g{\alpha}{\beta}R(g)
\end{equation*}
 with respect to the metric \eqref{metric::ansatz}. Straightforward computations lead to 
\begin{align*}
\Ein{t}{t} &= \frac{e^{2\mu(r)}}{r^2}\left( e^{-2\lambda(r)}\left(2r\lambda'(r) - 1 \right) + 1 \right) , \\
\Ein{r}{r} &= \frac{e^{2\lambda(r)}}{r^2}\left( e^{-2\lambda(r)}\left(2r\mu'(r) + 1 \right) - 1 \right) , \\
\Ein{\theta}{\theta} &= re^{-2\lambda(r)}\left( r\mu'^2(r) - \lambda'(r) + \mu'(r) - r\mu'(r)\lambda'(r) + r\mu''(r) \right),\\
\Ein{\phi}{\phi} &=  \sin^2\theta \Ein{\theta}{\theta},
\end{align*}
while the remaining components vanish. Therefore, the Einstein-Vlasov system is reduced to the following system of differential equations with respect the radial variable $r$:
\begin{align}
\label{Gtt}
e^{-2\lambda(r)}\left(2r\lambda'(r) - 1 \right) + 1 &= 8\pi^2e^{-2\mu(r)}\int_{D_r}\;E^2\Psi_\eta(r, (E, \ell, \mu)\Phi(E, \ell)\frac{\Phi(E, \ell))}{\sqrt{E^2 - E_{\ell}(r)}}\;dE d\ell ,\\
\label{Grr}
e^{-2\lambda(r)}\left(2r\mu'(r) + 1 \right) - 1 &= 8\pi^2e^{-2\mu(r)}\int_{D_r}\;\Phi(E, \ell)\Psi_\eta(r, (E, \ell, \mu)\sqrt{E^2 - E_{\ell}(r))}\;dE d\ell, \\
\left( r\mu'^2(r) - \lambda'(r) + \mu'(r) - r\mu'(r)\lambda'(r) + r\mu''(r) \right) &= \frac{4\pi^2}{r^3}e^{2\lambda(r)}\int_{D_r}\;  \Phi(E, \ell)\Psi_\eta(r, (E, \ell, \mu)\frac{\ell}{\sqrt{E^2 - E_{\ell}(r))}}\;dE d\ell.
\end{align}
We perform a last change of variable $\displaystyle E = e^{\mu(r)}\varepsilon$ and we set 
\begin{align}
\label{G::Phi}
G_\Phi(r, \mu) &:= \frac{2\pi}{r^2}\int_1^\infty\;\int_0^{r^2\left(\varepsilon^2 - 1\right)}\Phi(e^{\mu(r)}\varepsilon, \ell)\Psi_\eta(r, (e^{\mu(r)}\varepsilon, \ell, \mu))\frac{\varepsilon^2}{\sqrt{\varepsilon^2 - 1 - \frac{\ell}{r^2}}}\;d\ell d\varepsilon   , \\
\label{H::Phi}
H_\Phi(r, \mu) &:= \frac{2\pi}{r^2}\int_1^\infty\;\int_0^{r^2\left(\varepsilon^2 - 1\right)}\Phi(e^{\mu(r)}\varepsilon, \ell)\Psi_\eta(r, (e^{\mu(r)}\varepsilon, \ell, \mu))\sqrt{\varepsilon^2 - 1 - \frac{\ell}{r^2}}\;d\ell d\varepsilon.
\end{align}
The equations \eqref{Gtt} and \eqref{Grr} become
\begin{align}
\label{REV1}
e^{-2\lambda(r)}\left(2r\lambda'(r) - 1 \right) + 1 &= 8\pi r^2G_\Phi(r, \mu) ,\\
\label{REV2}
e^{-2\lambda(r)}\left(2r\mu'(r) + 1 \right) - 1 &= 8\pi r^2H_\Phi(r, \mu) . 
\end{align}
We call the latter nonlinear system for $\mu$ and $\lambda$ the \textit{reduced Einstein-Vlasov system}. 
\\We note that when inserting the ansatz of $f$ \eqref{f::ansatz} in the definition of the energy momentum tensor, the matter terms $G_\Phi$ and $H_\Phi$ become functionals of the yet unknown metric function $\mu$ and of the radial position $r$. Besides, from the reduced Einstein-Vlasov system \eqref{REV1}-\eqref{REV2}, we will show that one can express $\lambda$ in terms of the unknown metric $\mu$ so that we are left with the problem of solving only for $\mu$. Therefore, we will define the solution operator used in the implicit function theorem for $\mu$ only . For this, let $\rho>0$ and $R>0$ be  such that: 
\begin{equation*}
R> \max_{(E, \ell)\in B_{bound}} r_2^{Sch}(E, \ell),
\end{equation*}
\begin{equation*}
0< 2M + \rho< \min_{(E, \ell)\in B_{bound}} r_0^{Sch}(E, \ell)
\end{equation*}
and set $\displaystyle I:= ]2M+\rho, R[$. We will solve \eqref{REV1}-\eqref{REV2} on $I$. Then, we extend the solution to $]2M, \infty[$ by the Schwarzschild solution. In this context, we state the following lemma:
\begin{lemma}
\label{aux}
Let $\displaystyle \rho>0$ and $\displaystyle R> 2M+\rho$. Let $(\lambda, \mu)$ be a solution of the reduced Einstein Vlasov system such that $f$ is on the form  \eqref{f::ansatz}, with boundary conditions 
\begin{equation}
\label{lambda::0}
\lambda(\rho+2M)  = -\frac{1}{2}\log\left(1 - \frac{2M}{2M+\rho}\right) := \lambda_0,  
\end{equation}
and 
\begin{equation}
\label{mu::0}
\mu(\rho+2M)= \frac{1}{2}\log\left(1 - \frac{2M}{2M+\rho}\right) := \mu_0. 
\end{equation}
Moreover, we assume that 
\begin{equation}
\label{cond::m}
\displaystyle \forall r\in I\quad,\quad  2m(\mu)(r)<r.
\end{equation}
where
\begin{equation}
\label{mass}
m(\mu)(r) := M + 4\pi\int_{2M+\rho}^{r}s^2G_\Phi(s, \mu)\,ds. 
\end{equation}
Then, we have 
\begin{equation}
\label{lambda:delta}
e^{-2\lambda(r)} = 1 - \frac{2m(\mu)(r)}{r}.
\end{equation}
Besides,
\begin{equation}
\label{mu:delta}
\mu(r) = \mu_0 + \int_{2M+\rho}^{r}\frac{1}{1 - \frac{2m(\mu)(s)}{s}}\left(4\pi s H_{\Phi}(s, \mu) + \frac{1}{s^2}m(\mu)(s)\right)\,ds.
\end{equation} 
\end{lemma}
\begin{proof}
We integrate \eqref{REV1} between $2M+\rho$ and some $r\in I$ to obtain 
\begin{equation*}
\int_{2M+\rho}^r\;s(2\lambda'e^{-2\lambda(s)})\,ds - \int_{2M+\rho}^r\;e^{-2\lambda(s)}\,ds + (r - (2M+\rho)) = 8\pi\int_{2M+\rho}^r\;s^2G_\Phi(s, \mu)\,ds.
 \end{equation*}
We integrate by parts the first term of the left hand side to have 
 \begin{equation*}
e^{-2\lambda(r)} = 1 + \frac{2M+\rho}{r}\left( e^{-2\lambda(2M+\rho)} - 1\right) - \frac{8\pi}{r}\int_{2M+\rho}^r\;s^2G_\Phi(s, \mu)\,ds ,
 \end{equation*}
Since
\begin{equation*}
 e^{-2\lambda(2M+\rho)} = 1 - \frac{2M}{2M+\rho},
\end{equation*}
we have
\begin{equation}
\label{lambda}
\forall r\in I\,, \quad e^{-2\lambda(r)} = 1 - \frac{2m(\mu)(r)}{r}
\end{equation}
where 
\begin{equation}
\label{mass}
\forall r\in I\;,\; m(\mu)(r) := M + 4\pi\int_{2M+\rho}^{r}s^2G_\Phi(s, \mu)\,ds. 
\end{equation}
Thus, the solution of \eqref{REV1} is given by \eqref{lambda}. In order to find $\mu$, we use \eqref{lambda} so that we can write \eqref{REV2} on the form 
\begin{equation}
\label{mu::prime}
\mu'(r) =  e^{2\lambda(r)}\left( 4\pi r H_{\Phi}(r, \mu) + \frac{1}{r^2}m(\mu)(r) \right). 
\end{equation}
Now, we integrate the latter equation between $\displaystyle 2M+\rho$ and $\displaystyle r\in I$ to obtain 
\begin{equation}
\label{mu}
\mu(r) = \mu_0 + \int_{2M+\rho}^{r}\frac{1}{1 - \frac{2m(\mu)(s)}{s}}\left(4\pi s H_{\Phi}(r, \mu) + \frac{1}{s^2}m(\mu)(s)\right)\,ds.
\end{equation}
\end{proof}

\subsection{Rein's work on solutions with a Schwarzschild-like black hole}
\label{Rein:work}
In order to contrast it with our approach, we give an overview of Rein's proof \cite{rein1994static} concerning the construction of static spherically symmetric solutions to the Einstein-Vlasov system with a Schwarzschild-like black hole such that the spacetime has a finite mass and the matter field has a finite radius. Firstly, one starts with an ansatz of the form 
\begin{equation}
\label{ansatz::rein}
f(x, p) = \Phi(E, \ell) = \phi(E)(\ell - \ell_0)_+^{l}, \quad E > 0,\quad \ell\geq0,
\end{equation}
where $\ell_0\geq 0,\; l>1/2$ and $\phi\in L^\infty(]0, \infty[)$ is positive with $\phi(E) = 0$, $\forall E>E_0$ for some $E_0 > 0$.
\\ In this context, we note that  a necessary condition to obtain a steady state for the Einstein-Vlasov system with finite total mass is that $\Phi$ must vanish for energy values larger than some cut-off energy $E_0\geq 0$ \cite{rein2000compact}. The same result was proven in \cite{batt1986stationary} for the Vlasov-Poisson system. This motivates the choice of the cut-off function $\phi$.
\\ Under the above ansatz with a metric on the form \eqref{metric::ansatz}, the Einstein Vlasov system becomes
\begin{align}
\label{EV1:::}
e^{-2\lambda(r)}\left(2r\lambda'(r) - 1 \right) + 1 &= 8\pi r^2\overline G_\Phi(r,\mu(r)) ,\\
\label{EV2:::}
e^{-2\lambda(r)}\left(2r\mu'(r) + 1 \right) - 1 &= 8\pi r^2\overline H_\Phi(r,\mu(r)) , \\
e^{-2\lambda(r)}\left( \mu'' + (\mu' + \frac{1}{r})(\mu' - \lambda') \right) &= 8\pi \overline K_\Phi(r,\mu(r))
\end{align}
where
\begin{align}
\label{G::Phi::}
\overline G_{\Phi}(r, u) &= c_lr^{2l}e^{-(2l+4)u}g_\phi\left(e^u\sqrt{1+\frac{\ell_0}{r^2}}  \right), \\
\label{H::Phi::}
\overline  H_{\Phi}(r, u) &= \frac{c_l}{2l + 3}r^{2l}e^{-(2l+4)u}h_\phi\left(e^u\sqrt{1+\frac{\ell_0}{r^2}}  \right), \\
\label{K::Phi}
\overline  K_{\Phi}(r, u) &= (l+1)H_\Phi(r, u) + \frac{c_l}{2}\ell_0r^{2l-2}e^{-(2l+2)u}k_\phi\left(e^u\sqrt{1+\frac{\ell_0}{r^2}}  \right), \\
\end{align}
where $g_\phi$, $h_\phi$ and $k_\phi$ are defined by 
\begin{align}
\label{g::Phi}
g_\phi(t) &:= \int_t^\infty\;\phi(\varepsilon)\varepsilon^2(\varepsilon^2 - t^2)^{l+\frac{1}{2}}d\varepsilon, \\
\label{h::Phi}
h_\phi(t) &:= \int_t^\infty\;\phi(\varepsilon)(\varepsilon^2 - t^2)^{l+\frac{3}{2}}d\varepsilon, \\
\label{k::Phi}
k_\phi(t) &:= \int_t^\infty\;\phi(\varepsilon)(\varepsilon^2 - t^2)^{l+\frac{1}{2}}d\varepsilon \\
\end{align}
and $c_l$ is defined by 
\begin{equation}
\label{c:l}
c_l := 2\pi\int_0^1\frac{s^l}{\sqrt{1-s}}ds.
\end{equation}
More precisely, we have 
\begin{Propo}[Rein, \cite{rein1994static}]
\label{glob::existence}
Let $\Phi$ satisfy the assumptions stated above. Then for every $r_0\geq 0$,  $\lambda_0\geq 0 $ and $\mu_0 \in \mathbb R$ with $\lambda_0 = 0$ if $r_0 = 0$, there exists a unique solution $\lambda, \mu\in C^1([r_0, \infty[)$ of the reduced Einstein equations \eqref{EV1:::}-\eqref{EV2:::}  with 
\begin{equation*}
\lambda(r_0) = \lambda_0, \quad\quad  \mu(r_0) = \mu_0. 
\end{equation*} 
\end{Propo}
Local existence is first proven for $\displaystyle \lambda,\mu\in C^1([r_0, R[)$ where $R>r_0$. Then, solutions are shown to extend to $R=\infty$. To this end, we note the crucial use of the Tolman-Oppenheimer-Volkov (TOV) equation: 
\begin{equation}
\label{TOV::equation}
p'(r) = -\mu'(r)(p(r) + \rho(r)) - \frac{2}{r}(p(r) - p_T(r)),
\end{equation}
\begin{equation*}
\rho(r) = \overline G_{\Phi}(r, \mu(r)),\;\; p(r) = \overline H_{\Phi}(r, \mu(r)),\;\text{and}\; p_T(r) = \overline K_{\Phi}(r, \mu(r)). 
\end{equation*}
$\rho$, $p$ and $p_T$ are interpreted respectively as the energy density, the radial pressure and the tangential pressure. 
\begin{figure}[h!]
\includegraphics[width=\linewidth]{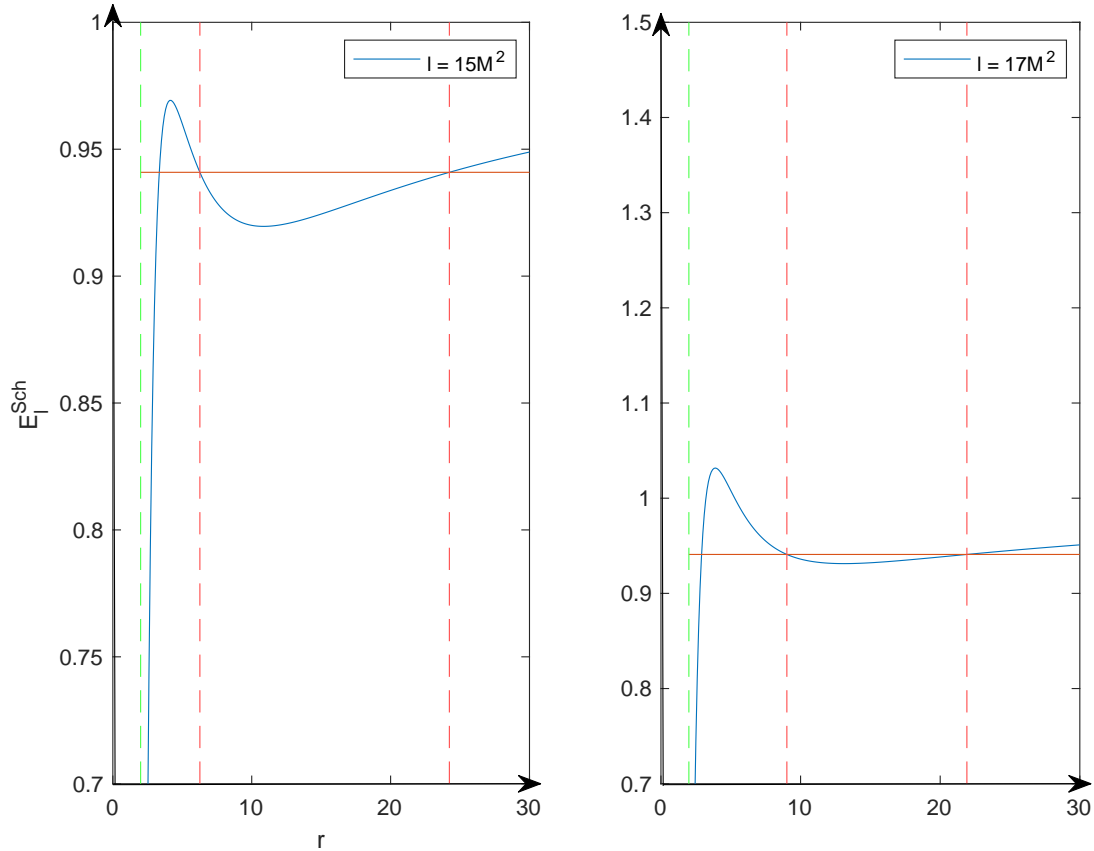}
\label{comparison}
\caption{The three roots of the equation $E_\ell^{Sch}(r) = E^2$ with $E = 0.97$ and $M=1$ for two cases of $\ell$: Left panel  $\ell < 16M^2$, which corresponds to the case where $E^{max}(\ell)< 1$. Right panel : $\ell>16M^2$, where  $E^{max}(\ell)> 1$.}
\end{figure}
In order to construct solutions outside from the Schwarzschild black hole, the assumption $\ell_0>0$ plays an important role in having vacuum between $2M$ and some $r_0> 2M$, to be defined below. The construction is then based on  gluing a vacuum region and a region containing Vlasov matter. A posteriori, one can check that the spacetime has a finite total ADM mass and the matter has a finite radius.
\\ For  the gluing,
\begin{itemize}
\item one starts by fixing a Schwarzschild black hole of mass $M$, and one then imposes a vacuum region until $r_0> 2M$. The position $r_0$ will be chosen in the following way:
\begin{enumerate}
\item First, one can  fix $E_0 = 1$. We note that, according to Proposition \ref{Propo1}, there can be no bounded orbits for $E_0>1$ in the Schwarzschild spacetime, which motivates this value of $E_0$\footnote{In the small data regime, one expects that the range of parameters leading to trapped geodesics to be close to that of Schwarzschild.}. 
This implies 
\begin{align*}
\overline G_{\Phi}(r, u) = \overline H_{\Phi}(r, u) = 0 \quad\quad\text{if } e^u\sqrt{1+ \frac{\ell_0}{r^2}} \geq 1, \\
\overline G_{\Phi}(r, u) , \overline H_{\Phi}(r, u) > 0 \quad\quad\text{if } e^u\sqrt{1+ \frac{\ell_0}{r^2}} < 1. \\
\end{align*}
In particular, the Schwarzschild metric with mass $M$ solves the reduced Einstein-Vlasov system for all $r>2M$ such that 
\begin{equation*}
\sqrt{1 - \frac{2M}{r}}\sqrt{1 + \frac{\ell_0}{r^2}} \geq 1.
\end{equation*}
This imposes $\ell_0> 16M^2$ and we have vacuum region when  $\displaystyle r\in[r_-, r_+]$, where
\begin{equation*}
r_\pm := \frac{\ell_0 \pm \sqrt{\ell_0^2 - 16M^2\ell_0}}{4M}.
\end{equation*}
Define 
\begin{equation}
\label{Abound:Rein}
\Abound^{Rein} = \left\{(E, \ell)\in]\sqrt{\frac{8}{9}}, 1[\times [\ell_0, \infty[ \right\}.
\end{equation}
Note that $\Abound^{Rein}$ is strictly included in $\Abound$. Indeed, there exists trapped geodesics of the Schwarzschild spacetime  such that $E<1$ and $\ell< 16 M^2$, cf Figure \ref{comparison} and Proposition \ref{Propo1}. In contrast, our solutions can be possibly supported on $\displaystyle \left]\sqrt{\frac{8}{9}}, 1\right[\times]12M^2, +\infty[$.
\item Now, one can impose the distribution function to be zero in the region $]2M, r_-[$ and extend the metric with a Schwarzschild metric up to $2M$. Note that in this context, the ansatz on $f$ is no longer valid for $r>2M$ so that the distribution function is not purely a function of $E$ and $\ell$.  This is related to our cut-off function $\Chi_\eta$. In fact, we recall from Proposition \ref{Propo1} that for a particle with $(E, \ell)\in\Abound$, two orbits are possible. The cut-off function selects only the trapped ones which are located in the region $r\geq r_1(E,\ell)$, "similar to" the region $r\geq r_+$ in Rein's work.   
\end{enumerate}
\item Next, one can impose initial data at $r_0 := r_+$ to be 
\begin{equation*}
\mu_0 := \sqrt{1 - \frac{2M}{r_0}}, \quad\quad \lambda_0 := \frac{1}{\sqrt{1 - \frac{2M}{r_0}}}. 
\end{equation*}
and apply Proposition \ref{glob::existence}. The solution to the Einstein Vlasov system is then obtained on $[r_0, \infty[$ and extended to the whole domain  by gluing it to the Schwarzschild metric at $r_0$. 
\end{itemize}

\section{Statement of the main result}
\label{main:result}
In this section, we give a more detailed formulation of our result. 
More precisely, we have 
\begin{theoreme}
\label{main::result}
Let $M>0$ and let $\mathcal O = \mathbb R\times]2M, \infty[ \times S^2$ be the domain of outer communications parametrised by the standard Schwarzschild coordinates $(t, r, \theta, \phi)$. Let $B_{bound}\subset\subset \mathcal A_{bound}$ be a compact subset of the set $\Abound$ defined in the following way:
\begin{itemize}
\item Fix $\displaystyle E_1, E_2\in\left(\sqrt{\frac{8}{9}}, 1\right)$ such that $E_1 < E_2$, 
\item Let $[\ell_1, \ell_2]$ be any compact subset of $]\ell_{lb}(E_2), \ell_{ub}(E_1)[$, where $\ell_{lb}$ and $\ell_{ub}$ are respectively defined by \eqref{l::lb} and \eqref{l::ub}.
\item Set 
\begin{equation}
\label{B::bound}
\displaystyle B_{bound} := [E_1, E_2]\times[\ell_1, \ell_2].
\end{equation} 
\end{itemize}
Let $\Phi:\Abound\times\mathbb R_+ \to \mathbb R_+$ be a $C^2$ function with respect to the first two variables, $C^1$ with respect to the third variable and such that
\begin{itemize}
\item $\displaystyle \forall\delta\in[0, \infty[\;$, $\displaystyle \Phi(\cdot, \cdot; \delta)$ is supported in $\displaystyle B_{bound}$.  
\item $\displaystyle \forall (E, \ell)\in\Abound\,,\;\; \Phi(E, \ell; 0) = \partial_\ell\Phi(E, \ell; 0) = 0$, does not identically vanish and $\displaystyle\forall\delta>0, \Phi(\cdot, \cdot, \delta)$ does not identically vanish on $B_{bound}$.   
\end{itemize}
Then, there exists $\delta_0>0$ and a one-parameter family of functions 
\begin{equation*}
(\lambda^\delta, \mu^\delta)_{\delta\in[0, \delta_0[}\in(C^2(]2M, \infty[))^2, \; f^\delta\in C^2(\mathcal O\times\mathbb R^3)
\end{equation*}
with the following properties 
\begin{enumerate}
\item $(\lambda^0, \mu^0) = (\lambda^{Sch}, \mu^{Sch})$ corresponds to a Schwarzschild solution with mass $M$. 
\item For all $(E, \ell)\in B_{bound}$, the equation 
\begin{equation*}
e^{\mu^\delta(r)}\left(1 + \frac{\ell}{r^2} \right) = E^2 
\end{equation*} 
admits three distinct positive roots $2M< r_0(\mu^\delta, E, \ell)< r_1(\mu^\delta, E, \ell) <r_2(\mu^\delta, E, \ell)$.  Moreover, there exists $\eta>0$ depending only on $\delta_0$ such that 
\begin{equation*}
r_0(\mu^\delta, E, \ell) + \eta < r_1(\mu^\delta, E, \ell).  
\end{equation*}
\item  The function $\displaystyle f^\delta$ takes the form 
\begin{equation*}
f^\delta(x,v) =\Phi(E^\delta, \ell; \delta)\Psi_\eta\left(r, (E^\delta, \ell), \mu^\delta\right), 
\end{equation*}
for $(x, v)\in\mathcal O\times \mathbb R^3$ with coordinates $(t, r, \theta, \phi, v^r, v^\theta, v^\phi)$ and where 
\begin{equation}
E^\delta := e^{2\mu^\delta(r)}\sqrt{1 + (e^{\lambda^\delta(r)}v^r)^2 + (rv^\theta)^2 + (r\sin\theta v^\phi)^2}\quad,\quad  \ell := r^4\left((v^\theta)^2 + \sin^2\theta(v^\phi)^2\right),
\end{equation}
and $\Psi_\eta$ is defined by \eqref{cut::off}-\eqref{cut:off:bis}. 
\item Let $g^\delta$ be defined by
\begin{equation}
\label{metric:::ansatz}
g^\delta_{(t, r, \theta, \phi)}= -e^{2\mu^{\delta}(r)} dt^2 + e^{2\lambda^{\delta}(r)} dr^2 + r^2(d\theta^2 + \sin^2\theta d\phi^2),
\end{equation}
then $(\mathcal O, g_\delta, f^\delta)$ is a static and spherically symmetric solution to the Einstein-Vlasov system \eqref{EFE} - \eqref{Vlasov2} - \eqref{EM_tensor} describing a  matter shell orbiting a Schwarzschild like black hole in the following sense: 
\begin{itemize}
\item $\displaystyle \exists R_{min}^\delta, R_{max}^\delta \in]2M,\infty[$ called the matter shell radii, which satisfy 
\begin{equation*}
R_{min}^\delta < R_{max}^\delta, \quad\text{and}\quad R_{min}^\delta>4M. 
\end{equation*}
\item $\forall r\in]2M, R_{min}^\delta[\cup]R_{max}^\delta, \infty[ ,$
\begin{equation*}
f^\delta(x, v) = 0,
\end{equation*}
\item $\exists \tilde r\in]R_{min}^\delta, R_{max}^\delta[ ,$
\begin{equation*}
f^\delta(\tilde r, \cdot) > 0.   
\end{equation*}
\item The metric $g^\delta$ is given by the  Schwarzschild metric with mass $M$ in the region $]2M, R_{min}^\delta[$.
\item $\exists M_\delta >M$ such that the metric $g^\delta$ is again given by the  Schwarzschild metric with mass $\displaystyle M^\delta$ in the region $]R_{max}^\delta, \infty[$. 
\end{itemize}
\end{enumerate}
\end{theoreme}

\section{Solving the reduced Einstein Vlasov system}
\label{Set-up}
\subsection{Set up for the implicit function theorem}
In this section, we define the solution mapping on which we are going to apply the implicit function theorem.
\begin{theoreme}[Implicit function theorem \footnote{See Theorem 17.6 , Ch. 17 of \cite{gilbarg2015elliptic} for a proof.}]
\label{IFT}
Let $\mathcal B_1$, $\mathcal B_2$ and $X$ be Banach spaces and $G$ a mapping from an open subset $\mathcal U$ of $\mathcal B_1\times X$ into $\mathcal B_2$. Let $(u_0, \sigma_0)$ be a point in $\mathcal U$ satisfying 
\begin{enumerate}
\item $G[u_0, \sigma_0] = 0$,
\item $G$ is continuously Fréchet differentiable on  $ \mathcal U$,
\item the partial Fréchet derivative with respect to the first variable  $L = G^1_{(u_0, \sigma_0)}$ is invertible.
\end{enumerate} 
Then, there exists a neighbourhood $\mathcal N$ of $\sigma_0$ in $X$ such that the equation $G[u,\sigma] = 0$ is solvable for each $\sigma\in \mathcal N$, with solution $u = u_\sigma\in\mathcal B_1$. 
\end{theoreme}
We recall that a mapping $\displaystyle G : \mathcal U \subset \mathcal B_1\times X \to \mathcal B_2$ is said to be Fréchet differentiable at a point $(u, \sigma)\in\mathcal U$ if there exits a continuous linear map $L(u, \sigma): \mathcal B_1\times X\to \mathcal B_2$ such that 
\begin{equation*}
\lim_{||(\delta u, \delta\sigma)||_{\mathcal B_1\times X}\to 0}\;\frac{||G(u + \delta u, \sigma + \delta\sigma) - G(u, \sigma) - L(u, \sigma)\cdot(\delta u, \delta\sigma)||_{\mathcal B_2}}{||(\delta u, \delta\sigma) ||_{\mathcal B_1\times X}} = 0. 
\end{equation*} 
$G$ is  Fréchet differentiable  if it is Fréchet differentiable at every point $(u, \sigma)\in\mathcal U$. It is continuously Fréchet differentiable if the mapping 
\begin{align*}
L :&\;\mathcal U \to \mathcal L(\mathcal B_1\times X, \mathcal B_2) , \\
& (u, \sigma) \mapsto L(u, \sigma)
\end{align*}
is continuous. For every $(u, \sigma)\in \mathcal U$ such that $G$ is Fréchet differentiable, the map $L(u, \sigma)$ is called the Fréchet differential at $(u, \sigma)$ of $G$ and it is noted $DG_{(u, \sigma)}$. 
\\By the partial Fréchet derivatives of $G$, denoted $G^1_{(u, \sigma)}, G^2_{(u, \sigma)}$ at $(u, \sigma)$, we mean the bounded linear mappings from $\mathcal B_1$, $X$ respectively, into $\mathcal B_2$ defined by 
\begin{equation*}
G^1_{(u, \sigma)}(h) := DG_{(u, \sigma)}(h, 0) \;,\; G^2_{(u, \sigma)}(k):= DG_{(u, \sigma)}(0, k), 
\end{equation*} 
for $h\in \mathcal B_1$, $k\in X$.
\\ 
\\Solutions to the reduced Einstein-Vlasov system will be obtained by perturbing the Schwarzschild spacetime using a bifurcation parameter $\delta\geq 0$. The latter turns on in the presence of Vlasov matter supported on $\displaystyle B_{bound}\subset\subset \Abound$. To this end, we transform the problem of finding solutions to the differential equation \eqref{mu::prime} into the problem of finding zeros of an operator $G$, for which we will apply the implicit function theorem. The Schwarzschild metric and the parameter $\delta = 0$, more precisely  $(\displaystyle \mu^{Sch}, 0)$, will be a zero of this operator.
\\For this, we adjust the ansatz \eqref{f::ansatz} to make the dependence on $\delta$ explicit:
\begin{equation}
\label{f:ansatz}
f = \Phi(E, \ell; \delta)\Chi_\eta(r - r_1(\mu, E, \ell)),
\end{equation}
such that 
\begin{equation*}
\forall (E, \ell)\in \Abound, \quad \Phi(E, \ell;\, 0) = 0
\end{equation*}
where $\Phi : \Abound \times \mathbb R_+ \to \mathbb R_+$. 
We will impose in the following some regularity conditions on $\Phi$ so that the solution operator is well defined. Assuming that $(\lambda, \mu, f)$ solve the EV-system, we can apply Lemma \ref{aux} with the ansatz \eqref{f:ansatz} to obtain 
\begin{equation}
\label{lambda:delta}
e^{-2\lambda(r)} = 1 - \frac{2m(\mu; \delta)(r)}{r}, 
\end{equation}
where
\begin{equation}
\label{mass}
m(\mu; \delta)(r) := M + 4\pi\int_{2M+\rho}^{r}s^2G_\Phi(s, \mu;\delta)\,ds,
\end{equation}
\begin{equation}
\label{mu:prime}
\mu'(r) =  e^{2\lambda(r)}\left( 4\pi r H_{\Phi}(r, \mu; \delta) + \frac{1}{r^2}m(\mu; \delta)(r) \right)
\end{equation}
and 
\begin{equation}
\label{mu:delta}
\mu(r) = \mu_0 + \int_{2M+\rho}^{r}\frac{1}{1 - \frac{2m(\mu; \delta)(s)}{s}}\left(4\pi s H_{\Phi}(s, \mu; \delta) + \frac{1}{s^2}m(\mu; \delta)(s)\right)\,ds
\end{equation} 

where 
\begin{align}
\label{G::Phi}
G_\Phi(r, \mu; \delta) &:= \frac{\pi}{r^2}\int_1^\infty\;\int_0^{r^2\left(\varepsilon^2 - 1\right)}\Phi(e^{\mu(r)}\varepsilon, \ell; \delta)\Psi_\eta(r, (e^{\mu(r)}\varepsilon, \ell), \mu)\frac{\varepsilon^2}{\sqrt{\varepsilon^2 - 1 - \frac{\ell}{r^2}}}\;d\ell d\varepsilon   , \\
\label{H::Phi}
H_\Phi(r, \mu; \delta) &:= \frac{\pi}{r^2}\int_1^\infty\;\int_0^{r^2\left(\varepsilon^2 - 1\right)}\Phi(e^{\mu(r)}\varepsilon, \ell; \delta)\Psi_\eta(r, (e^{\mu(r)}\varepsilon, \ell), \mu)\sqrt{\varepsilon^2 - 1 - \frac{\ell}{r^2}}\;d\ell d\varepsilon.
\end{align}
As we mention before, we will apply Theorem \ref{IFT} to solve \eqref{mu:prime}. Once obtained, we deduce $\lambda$ and $f$ through \eqref{lambda:delta} and \eqref{f:ansatz}. We define now the function space in which we will obtain the solutions of \eqref{mu:prime}.  We consider the Banach space
\begin{equation}
\label{def::I}
\mathcal X := \left(C^1(\overline{I}), ||\cdot||_{C^1(\overline{I})}\right), \quad\quad I = ]2M+\rho, R[
\end{equation}
and we recall the definition of the $C^1$ norm on $C^1(\Ibarre)$. 
\begin{equation}
\forall g\in C^1(\Ibarre) \quad :\quad ||g||_{C^1(\Ibarre)} := ||g||_{\infty} + ||g'||_{\infty}. 
\end{equation}
We define  $\mathcal U (\tilde\delta_0)\subset \mathcal X\times[0, \tilde\delta_0[ $ for some $  \tilde\delta_0\in]0,  \delta_{max}[ $ as
\begin{equation}
\mathcal U(\tilde\delta_0) := \left\{ (\mu; \delta)\in\mathcal X\times[0, \tilde\delta_0[ \;:\; \left|\left|\mu-\mu^{Sch}\right|\right|_{C^1(\Ibarre)}< \tilde\delta_0  \right\} = B(\mu^{Sch}, \tilde\delta_0)\times[0, \tilde\delta_0[, 
\end{equation}
where $B(\mu^{Sch}, \tilde\delta_0)$ is the open ball in $\mathcal X$ of centre $\mu^{Sch}$ and radius $\tilde\delta_0$ and $\delta_{max}$ is defined by 
\begin{equation*}
 \delta_{max}:=  \min\left\{\min_{(E, \ell)\in B_{bound}} r_0^{Sch}(E, \ell) - (2M+\rho), R - \max_{(E, \ell)\in B_{bound}} r_2^{Sch}(E, \ell)\right\}.
\end{equation*}
$\tilde\delta_0$ will be chosen small enough so that the three roots of the equation \eqref{effective::potential} exist and so that the condition
\begin{equation}
%\label{cond::m}
\forall r\in I\;,\; \quad\quad2m(\mu; \delta)(r) < r 
\end{equation}
is satisfied.  Besides, we make the following assumptions on $\Phi$:
\begin{itemize}
\item[$(\Phi_1)$] $\displaystyle \forall \delta\in[0, \tilde\delta_0[, \;\; \text{supp}\,\Phi(\cdot\,;\delta)\subset B_{bound}$, 
\item[$(\Phi_2)$] $\displaystyle \Phi$ is $C^1$ with respect to $\displaystyle \delta$ and $\displaystyle \exists C> 0,\, \forall (E, \ell)\in B_{bound}, \;\forall \delta\in[0, \tilde\delta_0[, \quad |\Phi(E, \ell; \delta)|\leq C$,
\item[$(\Phi_3)$] $\displaystyle \Phi$ is $C^2$ with respect to $\displaystyle (E, \ell)$. 
\end{itemize}
In view of \eqref{mu:delta}, we define the solution operator $G$ corresponding to $\mu$ by
\begin{align*}
G:\mathcal U(\tilde\delta_0)&\to \mathcal X \\
(\tilde\mu; \delta)&\mapsto G(\tilde\mu; \delta)
\end{align*}
where $\forall r\in\Ibarre$, 
\begin{equation}
\label{G::mu}
G(\tilde\mu; \delta)(r) := \tilde \mu(r) -  \left( \mu_0 + \int_{2M+\rho}^{r}\frac{1}{1 - \frac{2m(\tilde\mu; \delta)(s)}{s}}\left(4\pi s H_{\Phi}(s, \tilde\mu; \delta) + \frac{1}{s^2}m(\tilde\mu; \delta)(s)\right)\,ds \right). 
\end{equation}
We verify the steps  allowing to apply Theorem \ref{IFT}.  In Subsection \ref{radii::matter} below, we show that we have a well-defined ansatz for $f$, that is the existence of a matter shell surrounding the black hole after a well chosen $\tilde\delta_0$. In Subsection \ref{matter::terms}, we investigate the regularity of the matter terms and in Subsection \ref{Condition::m::r}, we check that \eqref{cond::m} is always satisfied after possibly shrinking $\tilde\delta_0$.

\subsection{Radii of the matter shell}
\label{radii::matter}
In this section, we show that for a suitable choice of $\tilde\delta_0$, $\forall \mu\in B(\mu^{Sch}, \tilde\delta_0)$, there exists $R_{min}(\mu), R_{max}(\mu)\in I$ satisfying $\displaystyle R_{min}(\mu)<R_{max}(\mu)$,  such that  $\displaystyle \text{supp}_r\,f \subset [R_{min}(\mu), R_{max}(\mu)]$. We also show that the ansatz for $f$ \eqref{f::ansatz} is well defined. More precisely, we state the following result

\begin{Propo}
\label{matter::shell}
Let $ 0< \tilde\delta_0 < \delta_{max}$. Then, there exists $\displaystyle\delta_0\in]0, \tilde\delta_0]$ such that $\displaystyle \forall\mu\in B(\mu^{Sch}, \delta_0)$, $\displaystyle \forall(E, \ell)\in B_{bound}$, there exist unique $\displaystyle r_i(\mu, (E, \ell))\in B(r^{Sch}_i(E, \ell), \delta_0)\footnote{$B(r^{Sch}_i(E, \ell), \delta_0) = \left\{ r\in I\;:\; |r - r^{Sch}_i(E, \ell)| <\delta_0 \right\}$.}, i\in\left\{0, 1, 2\right\}$ such that $r_i(\mu, (E, \ell))$ solve the equation
\begin{equation}
\label{eq--}
e^{2\mu(r)}\left(1 + \frac{\ell}{r^2}\right) = E^2. 
\end{equation}  
Moreover, there are no other roots for the	above equation outside the balls $B(r_i^{Sch}(E, \ell), \delta_0)$.
\end{Propo}

\begin{remark}
A direct application of the implicit function theorem would yield the existence of the three roots but for neighbourhoods of $\mu^{Sch}$ which a priori may depend on $(E, \ell)$. Thus, we revisit its proof  in order to obtain bounds uniform in $(E, \ell)$.
\end{remark}

\begin{proof}
\begin{enumerate}
\item \underline{Existence and uniqueness}:
Let $(E, \ell)\in B_{bound}$. Consider the mapping 
\begin{align*}
F^{E, \ell}: B(\mu^{Sch}, \tilde\delta_0)\times I &\to  \mathbb R \\
(\quad \mu\quad;\quad r\quad)&\mapsto E_\ell(r) - E^2 = e^{2\mu(r)}\left( 1 + \frac{\ell}{r^2}\right) - E^2.
 \end{align*} 
We have 
\begin{itemize}
\item $\displaystyle F^{E, \ell}$ is continuously Fréchet differentiable on $B(\mu^{Sch}, \tilde\delta_0) \times I$. In fact, 
\begin{itemize}
\item[\ding{169}] $\displaystyle \forall \mu\in B(\mu^{Sch}, \tilde\delta_0), \;$ the mapping $\displaystyle r\mapsto F^{E, \ell}(\mu, r) $ is continuously differentiable on $I$ with derivative 
\begin{equation*}
\frac{\partial F^{E, \ell}}{\partial r}(\mu, r) = 2e^{2\mu(r)}\left(\left( 1 + \frac{\ell}{r^2}\right)\mu'(r) - \frac{\ell}{r^3}\right).
\end{equation*} 
\item[\ding{169}]  It is easy to see that the mapping $\overline F^{E, \ell}:\mu\mapsto \overline F^{E, \ell}[\mu](r):= F^{E, \ell}(\mu, r)$ is continuously Fréchet differentiable on $\displaystyle B(\mu^{Sch}, \tilde\delta_0)$ with Fréchet derivative:
\begin{equation*}
\forall\mu\in B(\mu^{Sch}, \tilde\delta_0)\; ,\; \forall\tmu\in\mathcal X,\quad D\overline F^{E, \ell}(\mu)[\tmu](r) = 2\tmu(r)e^{2\mu(r)}\left(1 + \frac{\ell}{r^2} \right). 
\end{equation*}
\item[\ding{169}] Since the Fréchet differentials with respect to each variable exist and they are continuous, $F^{E, \ell}$ is continuously Fréchet differentiable on $B(\mu^{Sch}, \tilde\delta_0)\times I$. 

\end{itemize}
\item The points $(\mu^{Sch}, r_0^{Sch}(E, \ell))$, $(\mu^{Sch}, r_1^{Sch}(E, \ell))$ and $(\mu^{Sch}, r_2^{Sch}(E, \ell))$ are zeros for $F^{E, \ell}$. 
\item The differential  of $F^{E, \ell}$ with respect to $r$ at these points does not vanish. Otherwise, the trajectory is circular which is not possible since  $(E, \ell)\in\Abound$.  
\end{itemize}
We define the following mapping on $I$:  $\forall\mu\in B(\mu^{Sch}, \tilde\delta_0),\;$ 
\begin{equation*}
F_\mu^{E, \ell}(r) := r - \phi(E, \ell)F^{E, \ell}(\mu, r), 
\end{equation*}
where $\phi$ is defined by 
\begin{equation*}
\phi(E, \ell) = \left(\frac{\partial F^{E, \ell}}{\partial r}(\mu^{Sch}, r_i^{Sch}(E, \ell))\right)^{-1}. 
\end{equation*}
We will show that there exists $0<\delta_0<\tilde\delta_0$, uniform in $(E, \ell)$ such that $\forall\mu\in B(\mu^{Sch}, \tilde\delta_0),$ $F_\mu^{E, \ell}$ is a contraction on $\overline{B}(r_i^{Sch}(E, \ell), \delta_0)$. 
\\ Since $F^{E, \ell}$ is differentiable with respect to $r$, we have 
\begin{equation*}
\forall r\in I\;,\quad \frac{\partial F_\mu^{E, \ell}}{\partial r}(r) =  1 - \phi(E, \ell)\frac{\partial F^{E, \ell}}{\partial r}(\mu, r). 
\end{equation*}
Since the mappings $(E, \ell)\mapsto r_i^{Sch}$ and $(E, \ell)\mapsto F^{E, \ell}(\mu, r)$ depend smoothly on $(E, \ell)$ and since $B_{bound}$ is compact, there exists $C>0$ such that $\forall (E, \ell)\in B_{bound}$
\begin{equation*}
\left| \phi(E, \ell)\right| \leq C. 
\end{equation*}  
Moreover,   there exists $\delta_0<\tilde\delta_0$ such that $\forall(E, \ell)\in B_{bound}\; ,\;\forall (\mu, r)\in B(\mu^{Sch},\delta_0)\times B(r_i^{Sch}(E, \ell), \delta_0)$, 
\begin{equation*}
\left| \frac{\partial F^{E, \ell}}{\partial r}(\mu, r) - \frac{\partial F^{E, \ell}}{\partial r}(\mu^{Sch}, r_i^{Sch}(E, \ell)) \right| \leq \frac{1}{2(C+1)}. 
\end{equation*}
In fact, 
\begin{align*}
&\frac{\partial F^{E, \ell}}{\partial r}(\mu, r) - \frac{\partial F^{E, \ell}}{\partial r}(\mu^{Sch}, r_i^{Sch}(E, \ell)) =  \\
&2e^{\mu(r)}\left(\left( 1 + \frac{\ell}{r^2}\right)\mu'(r) - \frac{\ell}{r^3} - \left( 1 + \frac{\ell}{\left(r_i^{Sch}(E, \ell)\right)^2}\right)\mu'(r_i^{Sch}(E, \ell)) + \frac{\ell}{\left(r_i^{Sch}(E, \ell)\right)^3} \right) \\
&+ 2\left(\left( 1 + \frac{\ell}{\left(r_i^{Sch}(E, \ell)\right)^2}\right)\mu'(r_i^{Sch}(E, \ell)) - \frac{\ell}{\left(r_i^{Sch}(E, \ell)\right)^3}\right)\left(e^{2\mu(r)} - e^{2\mu^{Sch}(r)}\right)
\end{align*}
Since $B_{bound}$ is compact and the different quantities depend continuously on $(E, \ell)$, there exists $C_0>0$ such that $\forall(E, \ell)\in B_{bound}$, 
\begin{align*}
\left| 2\left(\left( 1 + \frac{\ell}{\left(r_i^{Sch}(E, \ell)\right)^2}\right)\mu'(r_i^{Sch}(E, \ell)) - \frac{\ell}{\left(r_i^{Sch}(E, \ell)\right)^3}\right)\left(e^{2\mu(r)}- e^{2\mu^{Sch}(r)} \right)\right| &\leq C_0\left| e^{2\mu(r)} - e^{2\mu^{Sch}(r)} \right| \\
&\leq C_0C_1\tilde\delta_0. 
\end{align*}
Similarly, we have 
\begin{align*}
\left| \frac{\ell}{r^3} - \frac{\ell}{\left(r_i^{Sch}(E, \ell)\right)^3} \right| &= \ell\left| \frac{r^2 - 2rr_i^{Sch}(E, \ell) + \left(r_i^{Sch}(E, \ell)\right)^2}{r^3r_i^{Sch}(E, \ell)^3}\right|\left|r - r_i^{Sch}(E, \ell) \right| \leq C_0\tilde\delta_0
\end{align*}
and
\begin{align*}
\left|\left(1 + \frac{\ell}{r^2}\right)\mu'(r) - \left( 1 + \frac{\ell}{\left(r_i^{Sch}(E, \ell)\right)^2}\right)\mu'(r_i^{Sch}(E, \ell)) \right| &\leq \left| \left(1 + \frac{\ell}{r^2}\right)\left(\mu'(r) -  \mu'(r_i^{Sch}(E, \ell)\right)\right|  \\
&+ \left| \ell\mu'(r_i^{Sch}(E, \ell)\left(\frac{1}{r^2} - \frac{1}{\left(r_i^{Sch}(E, \ell)\right)^2}\right) \right|  \\
&\leq C_0\tilde\delta_0. 
\end{align*}
Moreover, $r\mapsto e^{2\mu(r)}$ is bounded on $I$.  Therefore, we choose $\delta_0$ so that 
\begin{equation*}
\left| \frac{\partial F^{E, \ell}}{\partial r}(\mu, r) - \frac{\partial F^{E, \ell}}{\partial r}(\mu^{Sch}, r_i^{Sch}) \right| \leq \frac{1}{2(C+1)}. 
\end{equation*}
Hence, we obtain 
\begin{equation*}
\left| \frac{\partial F_{\mu}^{E, \ell}}{\partial r}(r) - \frac{\partial F_{\mu^{Sch}}^{E, \ell}}{\partial r}(r_i^{Sch}) \right| \leq \frac{C}{2(C+1)} \leq \frac{1}{2}. 
\end{equation*}
Since $\displaystyle \frac{\partial F_{\mu^{Sch}}^{E, \ell}}{\partial r}(r_i^{Sch}(E, \ell)) = 0$, we get 
\begin{equation*}
\left| \frac{\partial F_{\mu}^{E, \ell}}{\partial r}(r) \right|\leq \frac{1}{2}.
\end{equation*}
Now, we show that 
\begin{equation*}
\forall r\in B(r_i^{Sch}(E, \ell), \delta_0)\;, \quad \left|F_{\mu}^{E, \ell}(r) - r_i^{Sch}(E, \ell)\right| \leq \delta_0. 
\end{equation*}
We have $\displaystyle \forall(E, \ell)\in B_{bound},\, \forall r\in B(r_i^{Sch}(E, \ell), \delta_0)$
\begin{align*}
F_{\mu}^{E, \ell}(r) - r_i^{Sch}(E, \ell) &= F_{\mu}^{E, \ell}(r) - F_{\mu^{Sch}}r_i^{Sch}  \\
&= F_{\mu}^{E, \ell}(r) - F_{\mu}(r_i^{Sch}(E, \ell)) + F_{\mu}^{E, \ell}(r_i^{Sch}(E, \ell)) - F_{\mu^{Sch}}(r_i^{Sch}(E, \ell)).
\end{align*}
By the mean value theorem, 
\begin{equation*}
\left| F_{\mu}^{E, \ell}(r) - F_{\mu}(r_i^{Sch}(E, \ell)) \right| \leq \frac{1}{2}|r - r^{Sch}_i(E, \ell)| \leq \frac{\delta_0}{2}.
\end{equation*}
Moreover, by the same argument above, we show that there exists $C>0$ independent from $(E, \ell)$ such that 
\begin{equation*}
\left|F_{\mu}^{E, \ell}(r_i^{Sch}) - F_{\mu^{Sch}}(r_i^{Sch}(E, \ell))\right| \leq C\delta_0. 
\end{equation*}
Therefore, we can update $\delta_0$ so that 
\begin{equation*}
\left|F_{\mu}^{E, \ell}(r) - r_i^{Sch} \right| \leq \delta_0. 
\end{equation*}
It remains to show that $F_\mu^{E, \ell}$ is a contraction on $\overline B(r_i^{Sch}(E, \ell), \delta_0),\;\; \forall\mu\in B(\mu^{Sch}, \delta_0)$. For this, we have by the mean value theorem,  
\begin{equation*}
\left|F_\mu^{E, \ell}(r_1) - F_\mu^{E, \ell}(r_2)\right| \leq \frac{1}{2}|r_1 - r_2|, \quad\forall r_1, r_2\in B(r_i^{Sch}(E, \ell), \delta_0). 
\end{equation*}
Thus, we can apply the fixed point theorem to obtain : there exists $\displaystyle \delta_0<\tilde\delta_0  $ such that 
\begin{equation*}
\forall(E, \ell)\in B_{bound}, \;\forall\mu\in B(\mu^{Sch}, \delta_0)\;\exists! r_i^{E, \ell}:B(\mu^{Sch}, \delta_0)\to B(r_i^{Sch}(E, \ell), \delta_0) \;\text{such that}\; F^{E, \ell}(\mu, r_i^{E,\ell}(\mu)) = 0.  
\end{equation*}
\item \underline{Regularity}: It remains to show that $\forall (E, \ell)\in B_{bound},$ the mapping $r_i^{(E, \ell)}$ is continuously Fréchet differentiable on $B(\mu^{Sch}, \delta_0)$. In order to lighten the expressions, we will not write the dependence of $r_i$ on $(E, \ell)$. 
\\First, we show that $r_i$ is Lipschitz. For this, let $\mu, \overline \mu\in B(\mu^{Sch}, \delta_0)$ and set
\begin{equation*}
r = r_i(\mu), \quad\quad \overline r =  r_i(\overline\mu). 
\end{equation*}
We have
\begin{align*}
r - \overline r &= r_i(\mu) - r_i(\overline\mu) \\
&= F_\mu(r) - F_{\overline\mu}(\overline r) \\
&= F_\mu(r) -  F_{\mu}(\overline r)  + F_\mu(\overline r) - F_{\overline\mu}(\overline r) \\
&= F_\mu(r) -  F_{\mu}(\overline r) + \phi(E, \ell)\left( F^{E, \ell}(\mu, \overline r) - F^{E, \ell}(\overline \mu, \overline r)\right). 
\end{align*}
We have 
\begin{equation*}
\left|F_\mu(r) -  F_{\mu}(\overline r)\right| \leq \frac{1}{2}\left| r - \overline r\right| 
\end{equation*}
and
\begin{align*}
\left| \phi(E, \ell)\left(F^{E, \ell}(\mu, \overline r) - F^{E, \ell}(\overline \mu, \overline r)\right)\right|\leq C\left|\left|\mu - \overline\mu\right|\right|_{\mathcal X}.
\end{align*}
Therefore,
\begin{equation*}
|r - \overline r| \leq 2C \left|\left|\mu - \overline\mu\right|\right|_{\mathcal X}. 
\end{equation*}
Thus, $r_i$ is Lipschitz, so continuous on $B(\mu^{Sch}, \delta_0)$. Since 
\begin{equation*}
\forall r\in B(r_i^{Sch}, \delta_0), \;\quad \left| \frac{\partial F_{\mu}^{E, \ell}}{\partial r}(r) \right|\leq \frac{1}{2}, 
\end{equation*}
$\displaystyle \mathlarger{\mathlarger{\mathlarger{\Sigma}}}\,\left(\frac{\partial F_{\mu}^{E, \ell}}{\partial r}(r) \right)^k$ converges to $\displaystyle \left( 1 - \frac{\partial F_{\mu}^{E, \ell}}{\partial r}(r) \right)^{-1} = \frac{1}{\phi(E, \ell)\frac{\partial F^{E, \ell}}{\partial r}}$. Hence, 
\begin{equation*}
\forall r\in B(r_i^{Sch}, \delta_0),\,\forall\mu\in B(\mu^{Sch}, \delta_0), \quad\quad  \frac{\partial F^{E, \ell}}{\partial r}(r, \mu) \neq 0. 
\end{equation*}
Since $F^{E, \ell}$ is differentiable at $(\overline\mu, \overline r)$, we have
\begin{equation*}
0 = F^{E, \ell}(\mu, r) - F^{E, \ell}(\overline\mu, \overline r)= D_\mu F^{E, \ell}(\overline\mu, \overline r)\cdot(\mu - \overline\mu) + \partial_rF^{E, \ell}(\overline\mu, \overline r)(r - \overline r) +  o\left(||\mu - \overline\mu||_{\mathcal X} + |r - \overline r|\right).  
\end{equation*}
By the above estimates we have 
\begin{equation*}
o(|r - \overline r|) = o (||\mu - \overline\mu||_{\mathcal X}).
\end{equation*}
Therefore, 
\begin{equation*}
 r - \overline r = - \left(\partial_r F^{E, \ell}(\overline\mu, \overline r)\right)^{-1}D_\mu F^{E, \ell}(\overline\mu, \overline r)\cdot(\mu - \overline\mu) + o (||\mu - \overline\mu||).
 \end{equation*}
\item Finally, we note that after updating $\delta_0$, we obtain 
\begin{equation}
\label{uniform::control}
\left|\left| E_\ell - E^{Sch}_\ell\right|\right|_{C^1(\Ibarre)} < \delta_0. 
\end{equation}
In fact, $\forall r\in\Ibarre$, 
\begin{align*}
E_\ell(r) - E_\ell^{Sch}(r) = \left(1 + \frac{\ell}{r^2}\right)\left(e^{2\mu}(r) - e^{2\mu^{Sch}}(r) \right) \\
\end{align*}
We have 
\begin{equation*}
\left| e^{2\mu}(r) - e^{2\mu^{Sch}}(r)\right| \leq C\delta_0, 
\end{equation*}
and $\displaystyle r\mapsto 1 + \frac{\ell}{r^2}$ is bounded. Therefore, we can control $\displaystyle ||E_\ell - E_\ell^{Sch}||_\infty$. 
Now, we compute $\forall r\in I$
\begin{align*}
E'_\ell(r) - {E^{Sch}_\ell}'(r) &= 2e^{2\mu(r)}\left(\mu'(r)\left(1 + \frac{\ell}{r^2}\right) - \frac{\ell}{r^3}\right)  - 2e^{2\mu^{Sch}(r)}\left({\mu^{Sch}}'(r)\left(1 + \frac{\ell}{r^2}\right) - \frac{\ell}{r^3}\right) \\
&= -\frac{2\ell}{r^3}\left(e^{2\mu(r)} - e^{2\mu^{Sch}(r)} \right) + 2\left(1 + \frac{\ell}{r^2}\right)\left( \mu'(r)e^{2\mu(r)} - {\mu^{Sch}}'(r)e^{2\mu^{Sch}(r)}\right) 
\end{align*}
We control the last term in the following way
\begin{align*}
\left|  \mu'(r)e^{2\mu(r)} - {\mu^{Sch}}'(r)e^{2\mu^{Sch}(r)} \right| &\leq   \left| \mu'(r)\left(e^{2\mu(r)} - e^{2\mu^{Sch}(r)}\right) \right| + \left|  \left( \mu'(r) - {\mu^{Sch}}'(r)\right)e^{2\mu^{Sch}(r)} \right|  \\
&\leq C\delta_0. 
\end{align*}
\item \underline{Uniqueness of the roots on $\Ibarre$:} We show, after possibly shrinking $\delta_0$, that $\forall\mu\in B(\mu^{Sch}, \delta_0)\,\forall (E, \ell)\in B_{bound},\;$ there are no others roots of the equation \eqref{eq--}. 
First, we set 
\begin{equation*}
\forall r\in\Ibarre\;,\quad P_{E, \ell}(r) := E_\ell(r) - E^2
\end{equation*}
and 
\begin{equation*}
\forall r\in\Ibarre\;,\quad P^{Sch}_{E, \ell}(r) := E^{Sch}_\ell(r) - E^2. 
\end{equation*}
We claim that for $\delta_0$ sufficiently small, there exists $\displaystyle C>0$ such that $\displaystyle \forall (E, \ell)\in B_{bound}\,,\, \forall r\in J := \Ibarre\backslash \bigcup_{i=0,1,2}B(r_i^{Sch}(E, \ell), \delta_0)$ we have  
\begin{equation*}
\left|P_{E, \ell}^{Sch}(r)\right| > C\delta_0. 
\end{equation*} 
In fact, by the monotonicity properties of $P_{E, \ell}^{Sch}$, it is easy to see that
\begin{align*}
 \displaystyle \forall (E, \ell)\in B_{bound}\,,\, \forall r\in J\,:\; 
&\left|P_{E, \ell}^{Sch}(r)\right| \geq \max\left\{\left|P_{E, \ell}^{Sch}(r_0^{Sch}(E, \ell)- \delta_0)\right|, P_{E, \ell}^{Sch}(r_0^{Sch}(E, \ell)+ \delta_0)\right. \\
&\left.  ,  P_{E, \ell}^{Sch}(r_1^{Sch}(E, \ell) - \delta_0), P_{E, \ell}^{Sch}(r_1^{Sch}(E, \ell) + \delta_0),  \right. \\
&\left. \left|P_{E, \ell}^{Sch}(r_2^{Sch}(E, \ell) - \delta_0)\right|, P_{E, \ell}^{Sch}(r_2^{Sch}(E, \ell) + \delta_0)\right\}.
\end{align*}
Now, we show that for all $|h|$ sufficiently small, there exist $C(h)> 0$ uniform in $(E, \ell)$ such that 
\begin{equation*}
|P_{E, \ell}^{Sch}(r_i^{Sch}(E, \ell)+ h)| > C(h). 
\end{equation*}
We have 
\begin{equation*}
P_{E, \ell}^{Sch}(r_i^{Sch}(E, \ell)+ h) = P_{E, \ell}^{Sch}(r_i^{Sch}(E, \ell)) + h(P_{E, \ell}^{Sch})'(r_i^{Sch}(E, \ell)) + o(h),
\end{equation*}
where $o(h)$ is uniform in $(E, \ell)$ by continuity of $(P_{E, \ell}^{Sch})^{(k)}$ with respect to $(E, \ell)$ and compactness of $B_{bound}$. Moreover,  $r_i^{Sch}(E, \ell)$ are simple roots so that $(P_{E, \ell}^{Sch})'(r_i^{Sch}(E, \ell)) \neq 0$ and $\displaystyle (E, \ell)\mapsto (P_{E, \ell}^{Sch})'(r_i^{Sch}(E, \ell))$ is continuous. Hence, for $h$ sufficiently small, we obtain
\begin{equation*}
|P_{E, \ell}^{Sch}(r_i^{Sch}(E, \ell)+ h)| > |h|\left|(P_{E, \ell}^{Sch})'(r_i^{Sch}(E, \ell))\right|> C|h|, 
\end{equation*}
where $C$ is some constant which uniform in $\delta_0$ and $(E, \ell)$. Therefore, we update $\delta_0$ so that 
\begin{equation*}
\left|P_{E, \ell}^{Sch}(r_i^{Sch}(E, \ell) \pm \delta_0)\right| > C\delta_0.
\end{equation*}  
Now, let $\delta_1 <\delta_0$. Then, $\displaystyle  \forall \mu\in B(\mu^{Sch}, \delta_1)\subset B(\mu^{Sch}, \delta_0) $,  $\displaystyle \forall (E, \ell)\in B_{bound},\; r_i(\mu, E, \ell)$ is unique in the ball $B(r^{Sch}_i(E, \ell), \delta_0)$.  Moreover, $\forall r\in J$, we have 
\begin{equation*}
P_{E, \ell}(r) = P_{E, \ell}(r) - P_{E, \ell}^{Sch}(r) + P_{E, \ell}^{Sch}(r).
\end{equation*}
By the triangular inequality, the latter implies for $\delta_1 < C\delta_0$ that 
\begin{equation*}
\forall r\in J, \quad\quad |P_{E, \ell}(r)| > -\delta_1 +  C\delta_0 > 0.
\end{equation*}
Therefore, after updating $\delta_1$, $P_{E, \ell}$ does not vanish outside the balls $B(r_i^{Sch}(E, \ell), \delta_1)$.  This yields the uniqueness. 
\end{enumerate}
\end{proof}

\begin{lemma}
$\displaystyle\forall\mu\in B(\mu^{Sch}, \delta_0)$, $r_i(\mu,\cdot)$ have the same monotonicity properties as $r_i^{Sch}$. More precisely, $\forall (E, \ell)\in B_{bound}, $
\begin{itemize}
\item $\displaystyle r_0(\mu, E, \cdot)$ decreases  on $\left]\ell_1, \ell_2\right[$ and  $\displaystyle r_0(\mu, \cdot, \ell)$ increases   on $\left] E_1, E_2 \right[$, 
\item $\displaystyle r_1(\mu, E, \cdot)$ increases  on $\left]\ell_1, \ell_2\right[$ and $\displaystyle r_1(\mu, \cdot, \ell)$ decreases on $\left] E_1, E_2 \right[$, 
\item $\displaystyle r_2(\mu, E, \cdot)$ decreases  on $\left]\ell_1, \ell_2\right[$ and $\displaystyle r_2(\mu, \cdot, \ell)$ increases on $\left] E_1, E_2 \right[$. 
\end{itemize}
\end{lemma}
The proof of the latter lemma is straightforward. 

Therefore, using the above monotonicity properties of $\displaystyle r_1(\mu, \cdot)$ and $\displaystyle r_2(\mu, \cdot)$ one can define the radii of the matter shell in the following way:
\begin{equation}
\label{Rmin}
R_{min}(\mu) := r_1(\mu, E_2, \ell_1)
\end{equation}
and
\begin{equation}
\label{Rmax}
R_{max}(\mu) := r_2(\mu, E_2, \ell_1). 
\end{equation}
Since $\forall (E, \ell)\in \Abound ,\; r_1^{Sch}(E, \ell)> 4M$, we can update $\delta_0$ such that 
\begin{equation*}
\forall(E, \ell)\in B_{bound} \;,\; r_1(\mu, E, \ell)>4M.
\end{equation*} 
In particular, 
\begin{equation}
\label{Rmin::cond}
R_{min}(\mu) > 4M. 
\end{equation}
Now, we need to define $\eta>0$ appearing in the ansatz of $f$, \eqref{f::ansatz}. We choose $\eta$ > 0 independent of $(\mu, E, \ell)$ such that 
\begin{equation*}
r_1(\mu, E, \ell) - r_0(\mu, E, \ell) > \eta > 0.  
\end{equation*}
\begin{lemma}
There exists $\eta$ > 0 independent of $\mu, E, \ell$ such that for all $\mu\in B(\mu^{Sch}, \delta_0)$ and for all $(E, \ell)\in B_{bound}$ we have $\displaystyle r_1(\mu, E, \ell) - r_0(\mu, E, \ell) > \eta > 0.  $
\end{lemma}
\begin{proof}
For this, we set 
\begin{align*}
h^{Sch}:  B_{bound} &\to \mathbb R \\
\; (E, \ell)&\mapsto r_1^{Sch}(E, \ell) - r_0^{Sch}(E, \ell)
\end{align*}
By monotonicity properties of $r_i^{Sch}$, it is easy to see that $h^{Sch}(\cdot, \ell)$ is decreasing and $h^{Sch}(E,\cdot)$ is increasing.  Therefore, one can easily bound $h$ :
\begin{equation*}
\forall(E, \ell)\in B_{bound}\;,\; h^{Sch}(E, \ell) \geq h^{Sch}(E_2, \ell) \geq h^{Sch}(E_2, \ell_1)> 0.  
\end{equation*} 
Thus, $\displaystyle\exists\eta>0$ such that $\displaystyle\forall (E, \ell)\in B_{bound}$
\begin{equation*}
h^{Sch}(E, \ell) >2\eta. 
\end{equation*}
Now we set 
\begin{align*}
h:  B(\mu^{Sch}, \delta_0)\times B_{bound} &\to \mathbb R \\
\; (\mu; (E, \ell))\quad\quad\quad&\mapsto h(\mu, (E, \ell)) := r_1(\mu, E, \ell) - r_0(\mu, E, \ell).
\end{align*}
We have, $\displaystyle\forall\mu\in B(\mu^{Sch}, \delta_0)$, $\forall (E, \ell)\in B_{bound}$
\begin{align*}
h(\mu, E, \ell) &= h(\mu, (E, \ell)) - h^{Sch}(E, \ell) +  h^{Sch}(E, \ell) > -2\delta_0 + 2\eta. 
\end{align*}
It remains to update $\delta_0$ so that the latter inequality is greater than $\eta$.
\end{proof}
Now that we have justified the ansatz for $f$, it remains to check that $\displaystyle \text{supp}_r\,f\subset[R_{min}(\mu), R_{max}(\mu)]$. To this end, we state the following result
\begin{lemma}
\label{supp:f}
Let $\mu\in B(\mu^{Sch}, \delta_0)$ and $f$ be a distribution function of the form \eqref{f::ansatz}. Then 
\begin{equation*}
\text{supp}_r\,f \subset [R_{min}(\mu), R_{max}(\mu)]. 
\end{equation*}
\end{lemma}
\begin{proof}
Let $(x^\mu, v^a)\in \Gamma$ and denote by $r$ its radial component. Let $E(x^\mu, v^a)$ and $\ell(x^\mu, v^a)$ be defined by \eqref{energy} and \eqref{ang::momentum}.
If $r\in \text{supp}_r\,f$, then 
\begin{equation*}
\Phi(E, \ell) >0, \quad\text{and}\quad \Psi_\eta(r, E, \ell, \mu) >0. 
\end{equation*}
Therefore, $(E, \ell)\in B_{bound}$ and 
\begin{equation*}
\Psi_\eta(r, E, \ell, \mu) = \Chi_\eta(r - r_1(\mu, (E, \ell))) > 0. 
\end{equation*}
Hence, 
\begin{equation*}
r \geq r_1(\mu, E, \ell) - \eta > r_0(\mu, E, \ell). 
\end{equation*}
The last inequality is due to the definition of $\eta$. Now since $(E, \ell)\in B_{bound}$, the equation $E_\ell(\tilde r) = E^2$ admits three distinct positive roots. Moreover, we have
\begin{equation}
\label{pot::ineq}
E_\ell(r) \leq E^2
\end{equation}
for any geodesic moving in the exterior region. Therefore, 
\begin{equation*}
r\in\left]2M+\rho, r_0(\mu, E, \ell)\right]\cup\left[r_1(\mu, E, \ell), r_2(\mu, E, \ell) \right]. 
\end{equation*} 
Hence, $r$ must lie in the region $\left[r_1(\mu, E, \ell), r_2(\mu, E, \ell) \right]$. By  construction of $R_{min}(\mu)$ and $R_{max}(\mu)$, we conclude that 
\begin{equation*}
r\in \left[R_{min}(\mu), R_{max}(\mu)\right]. 
\end{equation*}
\end{proof}
%In order to justify the choice of \eqref{mu::0} and \eqref{lambda::0}, it suffices to remark, using the above arguments, that there is vacuum in a neighbourhood of the point $2M+\rho$ so that $\mu$ is given by $\mu^{Sch}$ in this region. Moreover, the extension of solutions to $]2M, \infty[$ by Schwarzschild solution is straightforward.   
\subsection{Regularity of the matter terms}
\label{matter::terms}
In this section, we show that the matter terms $G_\phi$ and $H_\Phi$ given by respectively  \eqref{G::Phi} and \eqref{H::Phi} are well defined on $\displaystyle I\times\mathcal U(\delta_0)$. Then, we investigate their regularity with respect to each variable.  
\\ Let $(r, \mu; \delta) \in \displaystyle I\times \mathcal U(\delta_0)$ and let $\varepsilon\in[1,\infty[$ and $\ell\in[0, r^2(\varepsilon^2-1)[$. Note that  if $\displaystyle r^2(\varepsilon^2 - 1) < \ell_1$ , then $\ell<\ell_1$ and
\begin{equation*}
\Phi(e^{\mu(r)}\varepsilon, \ell; \delta) = 0,
\end{equation*}
since $\text{supp}_\ell\Phi\subset[\ell_1, \ell_2]$.  Therefore,
\begin{align*}
\displaystyle G_\Phi(r, \mu; \delta) &=  \frac{2\pi}{r^2}\int_1^\infty\;\int_0^{r^2\left(\varepsilon^2 - 1\right)}\Phi(e^{\mu(r)}\varepsilon, \ell; \delta)\Psi_\eta(r, (e^{\mu(r)}\varepsilon, \ell), \mu)\frac{\varepsilon^2}{\sqrt{\varepsilon^2 - 1 - \frac{\ell}{r^2}}}\;d\ell d\varepsilon   , \\
\displaystyle &=  \frac{2\pi}{r^2}\int_{\sqrt{1 + \frac{\ell_1}{r^2}}}^{e^{-\mu(r)}E_2}\;\int_{\ell_1}^{r^2\left(\varepsilon^2 - 1\right)}\Phi(e^{\mu(r)}\varepsilon, \ell; \delta)\Psi_\eta(r, (e^{\mu(r)}\varepsilon, \ell), \mu)\frac{\varepsilon^2}{\sqrt{\varepsilon^2 - 1 - \frac{\ell}{r^2}}}\;d\ell d\varepsilon.
\end{align*}
We make a first change of variable from  $\ell$ to $\tilde \ell := \ell - \ell_1$,
\begin{equation*}
\displaystyle G_\Phi(r, \mu; \delta) =  \frac{2\pi}{r^2}\int_{\sqrt{1 + \frac{\ell_1}{r^2}}}^{e^{-\mu(r)}E_2}\;\int_{0}^{r^2\left(\varepsilon^2 - 1\right) - \ell_1}\Phi(e^{\mu(r)}\varepsilon, \tilde\ell + \ell_1 ; \delta)\Psi_\eta(r, (e^{\mu(r)}\varepsilon, \tilde\ell+\ell_1), \mu)\displaystyle \frac{\varepsilon^2}{\sqrt{\left(\varepsilon^2 - 1 - \frac{\ell_1}{r^2}\right) - \frac{\tilde\ell}{r^2}}}\;d\tilde\ell d\varepsilon. 
\end{equation*}
We make a second change of variable from $\tilde\ell$ to $\displaystyle s := \frac{\tilde\ell}{r^2(\varepsilon^2 - 1) - \ell_1}$, 
\begin{equation*}
\displaystyle G_\Phi(r, \mu; \delta) =  2\pi\int_{\sqrt{1 + \frac{\ell_1}{r^2}}}^{e^{-\mu(r)}E_2}\;\varepsilon^2\left( \varepsilon^2 - \left(1 +  \frac{\ell_1}{r^2}\right)\right)^{\frac{1}{2}}g_\Phi(r, \mu,e^{\mu(r)} \varepsilon;\delta)d\varepsilon,
\end{equation*}
where $g_\Phi$ is defined by 
\begin{equation}
\label{g:phi}
g_\Phi(r, \mu, E;\delta) := \int_{0}^{1}\Phi(E, sr^2(e^{-2{\mu(r)}}E^2-1) + (1 - s)\ell_1 ; \delta)\Psi_\eta(r, E, sr^2(e^{-2{\mu(r)}}E^2-1) + (1 - s)\ell_1), \mu)\frac{ds}{\sqrt{1-s}}.
\end{equation}
We make a last change of variable from $\displaystyle \varepsilon$ to $E := e^{\mu(r)}\varepsilon$ to obtain 
\begin{equation}
\label{G::Phi::A}
\displaystyle G_\Phi(r, \mu; \delta) = 2\pi e^{-4{\mu(r)}}\int_{e^{\mu(r)}\sqrt{1 + \frac{\ell_1}{r^2}}}^{E_2}\;E^2\left(E^2 - e^{2{\mu(r)}}\left(1 +  \frac{\ell_1}{r^2}\right)\right)^{\frac{1}{2}}g_\Phi(r, \mu, E;\delta)dE. 
\end{equation}
With the same change of variables we compute
\begin{align*}
\displaystyle H_\Phi(r, \mu; \delta) &= \frac{2\pi}{r^2}\int_1^\infty\;\int_0^{r^2\left(\varepsilon^2 - 1\right)}\Phi(e^{{\mu(r)}}\varepsilon, \ell; \delta)\Psi_\eta(r, (e^{\mu(r)}\varepsilon, \ell), \mu)\sqrt{\varepsilon^2 - 1 - \frac{\ell}{r^2}}\;d\ell d\varepsilon \\
&= \frac{2\pi}{r^2}\int_{\sqrt{1 + \frac{\ell_1}{r^2}}}^{e^{-\mu(r)}E_2}\;\int_{\ell_1}^{r^2\left(\varepsilon^2 - 1\right)}\Phi(e^{\mu(r)}\varepsilon, \ell; \delta)\Psi_\eta(r, (e^{\mu(r)}\varepsilon, \ell), \mu)\sqrt{\varepsilon^2 - 1 - \frac{\ell}{r^2}}\;d\ell d\varepsilon  \\
&= \frac{2\pi}{r^2}\int_{\sqrt{1 + \frac{\ell_1}{r^2}}}^{e^{-\mu(r)}E_2}\;\int_{0}^{r^2\left(\varepsilon^2 - 1\right) - \ell_1}\Phi(e^{{\mu(r)}}\varepsilon, \tilde\ell+\ell_1; \delta)\Psi_\eta(r, (e^{\mu(r)}\varepsilon, \tilde\ell+\ell_1), \mu)\sqrt{\left(\varepsilon^2 - 1 - \frac{\ell_1}{r^2}\right) - \frac{\tilde\ell}{r^2}}\;d\tilde\ell d\varepsilon  \\
&= 2\pi\int_{\sqrt{1 + \frac{\ell_1}{r^2}}}^{e^{-\mu(r)}E_2}\;\left( (\varepsilon^2 - \left(1 +  \frac{\ell_1}{r^2}\right)\right)^{\frac{3}{2}}h_\Phi(r, \mu, e^{{\mu(r)}}\varepsilon;\delta)d\varepsilon,
\end{align*}
where $h_\Phi$ is defined by 
\begin{equation}
\label{h:phi}
h_\Phi(r, \mu, E;\delta) := \int_{0}^{1}\Phi(E, sr^2(e^{-2{\mu(r)}}E^2-1) + (1 - s)\ell_1 ; \delta)\Psi_\eta(r, (E, sr^2(e^{-2{\mu(r)}}E^2-1) + (1 - s)\ell_1), \mu)\sqrt{1-s}\,ds.
\end{equation}
Therefore,
\begin{equation}
\label{H::Phi::A}
\displaystyle H_\Phi(r,\mu; \delta) = 2\pi e^{-4\mu(r)}\int_{e^{\mu(r)}\sqrt{1 + \frac{\ell_1}{r^2}}}^{E_2}\;\left( E^2 - e^{2\mu(r)}\left(1 +  \frac{\ell_1}{r^2}\right)\right)^{\frac{3}{2}}h_\Phi(r, \mu, E;\delta)\,dE.
\end{equation}
It is clear that $h_\phi$ is well defined on $\displaystyle I\times C^1\times \mathcal U(\delta_0)$. Since $\displaystyle s\mapsto \frac{1}{\sqrt{1-s}}$ is integrable, $g_\phi$ is well defined on $\displaystyle I\times C^1\times \mathcal U(\delta_0)$. Now we state a regularity result of $G_\Phi$ and $H_\Phi$. 
\begin{Propo}
\label{regularity::G::H}
\begin{enumerate}
\item $g_\Phi$ and $h_\Phi$ defined respectively by \eqref{g:phi} and \eqref{h:phi} are $C^2$ with respect to $r$ and $E$ respectively on $I$, $]E_1, E_2[$ and $C^1$ with respect to $\delta$ on $[0, \delta_0[$. Furthermore, they are continuously Fréchet differentiable with respect to $\mu$ on $\displaystyle B(\mu^{Sch}, \delta_0)$. 
\item Similarly, $G_\Phi$ and $H_\Phi$ are continuously Fréchet differentiable with respect to $\mu$ on $\displaystyle B(\mu^{Sch}, \delta_0)$, $C^2$ with respect to $r$ and $C^1$ with respect to $\delta$.  
\end{enumerate}
\end{Propo}
\begin{proof}
The proof is based on the regularity of the ansatz function as well as on the regularity of $r_1$. 
\begin{enumerate}
\item 
\begin{itemize}
\item By  the dominated convergence theorem, regularity of $\Phi$, its compact support and the assumption $\displaystyle (\Phi_2)$, it is easy to see that $g_\Phi$ and $h_\Phi$ are $C^2$ with respect to $r$ and $E$ on their domain and $C^1$ with respect to $\delta$.
\item For the differentiability of $g_\Phi$ and $h_\Phi$ with respect to $\mu$:  let $(r, E, \delta)\in I\times]E_1, E_2[\times [0, \delta[$ and let $\mu\in B(\mu^{Sch}, \delta_0)$, we have
\begin{align*}
&g_\Phi(r, \mu, E;\delta) =  \\
&\int_{0}^{1}\Phi(E, sr^2(e^{-2{\mu(r)}}E^2-1) + (1 - s)\ell_1 ; \delta)\Psi_\eta(r, (E, sr^2(e^{-2{\mu(r)}}E^2-1) + (1 - s)\ell_1), \mu)\frac{ds}{\sqrt{1-s}} = \\
&\int_{0}^{1}\Phi(E, sr^2(e^{-2{\mu(r)}}E^2-1) + (1 - s)\ell_1 ; \delta)\Chi_\eta(r -  r_1(\mu, E, sr^2(e^{-2{\mu(r)}}E^2-1) + (1 - s)\ell_1))\frac{ds}{\sqrt{1-s}}.
\end{align*}
Since $r_1$ is continuously Fréchet differentiable on $\displaystyle B(\mu^{Sch}, \delta_0)$, $\Chi_\eta$ is smooth on $\mathbb R$ and $\Phi(\cdot, \cdot; \delta)$ is $C^2$ on $B_{bound}$, the mappings 
\begin{equation*}
\phi:\mu\mapsto  \Phi(E, sr^2(e^{-2{\mu(r)}}E^2-1) + (1 - s)\ell_1 ; \delta) =:\overline \Phi(s, \mu)
\end{equation*}
and 
\begin{equation*}
\psi: \mu\mapsto \Chi_\eta(r -  r_1(\mu, E, sr^2(e^{-2{\mu(r)}}E^2-1) + (1 - s)\ell_1))
\end{equation*}
are continuously Fréchet differentiable on $B(\mu^{Sch}, \delta_0)$. Their Fréchet derivatives are respectively given by : $\forall \tmu\in \mathcal X$, 
\begin{equation*}
D\phi(\mu)[\tmu] = -2sr^2E^2\tmu(r)\partial_\ell\Phi(E, sr^2(e^{-2{\mu(r)}}E^2-1) + (1 - s)\ell_1 ; \delta)e^{-2\mu(r)}
\end{equation*}
and 
\begin{align*}
D\psi(\mu)[\tmu] &= - \left(D_\mu\overline r_1(s, \mu)[\tmu]  - 2sr^2E^2\tmu(r)\partial_\ell\overline r_1(s, \mu)e^{-2\mu(r)}  \right)\Chi'_\eta(r -  \overline r_1(s, \mu))
\end{align*}
where 
\begin{equation*}
\overline r_1(s, \mu) := r_1(\mu, e^{\mu(r)}\varepsilon, sr^2(e^{-2{\mu(r)}}E^2-1) + (1 - s)\ell_1).
\end{equation*}
Since $\Chi_\eta$ is either $0$ or $1$ on the support of $\Phi(E, \ell; \delta)$, we have 
\begin{equation*}
D\psi(\mu)[\tmu] = 0. 
\end{equation*}
Therefore, for all $\tmu\in\mathcal X$ sufficiently small,   
\begin{align*}
g_\Phi(r, \mu+\tmu; \delta)(r) &= \int_0^1\overline\Phi(s, \mu + \tmu)\psi(\mu + \tmu)(s)\,ds \\
&= \int_0^1\left( \phi(\mu)(s) + D\phi(\mu)[\tmu](s) + O\left(||\tmu||^2_{C^1(\Ibarre)}\right)\right)\left(\psi(\mu)(s) + O\left(||\tmu||^2_{C^1(\Ibarre)}\right) \right)\,ds \\
&= \int_0^1\overline\Phi(s)\psi(\mu)(s) + \psi(\mu)(s)D\phi(\mu)[\tmu](s) + O\left(||\tmu||^2_{C^1(\Ibarre)}\right)\,ds \\
&= \g_\Phi(\mu)(r) + \int_0^{1}\;\psi(\mu)(s)D\phi(\mu)[\tmu](s)\frac{ds}{\sqrt{1-s}} + O\left(||\tmu||^2_{C^1(\Ibarre)}\right)\,ds
\end{align*}
Note that we didn't write the dependence of $\phi, \psi$ and $g_\Phi$ on the remaining variables in order to lighten the expressions. Now we define 
\begin{equation*}
Dg_\Phi(\mu)[\tmu](r) := -\int_0^{1}\;\psi(\mu)(s)D\phi(\mu)[\tmu](s)\frac{ds}{\sqrt{1-s}}.
\end{equation*}
It is clear that $\displaystyle Dg_\Phi(\mu)$ is a linear mapping from $\displaystyle \mathcal X$ to $\displaystyle \mathbb R$. Besides, since $r_1$ is continuously Fréchet differentiable on  $\displaystyle B(\mu^{Sch}, \delta_0)$,  $g_\Phi$ is also continuously  Fréchet differentiable.  
\item In the same way, we obtain regularity for $h_\Phi$. 
\end{itemize}
\item Regularity for $G_\Phi$ and $H_\Phi$ are straightforward.  In particular, their Fréchet derivatives with respect to $\mu$ are respectively given by $\displaystyle \forall\tmu\in\mathcal X$
\begin{equation}
\label{D3G:Phi}
\begin{aligned}
&D_\mu G_\Phi(r, \mu; \delta)[\tmu] =  2\pi e^{-4\mu(r)}\left( -4\tmu(r) \int_{e^{\mu(r)}\sqrt{1 + \frac{\ell_1}{r^2}}}^{E_2}\;E^2\left(E^2 - e^{2\mu(r)}\left(1 +  \frac{\ell_1}{r^2}\right)\right)^{\frac{1}{2}}g_\Phi(r,\mu, E; \delta)dE\right.   \\
&\left. +  \int_{e^{\mu(r)}\sqrt{1 + \frac{\ell_1}{r^2}}}^{E_2}\;-\frac{E^2\tmu(r)e^{2\mu(r)}\left(1 + \frac{\ell_1}{r^2} \right)}{\sqrt{E^2 - e^{2\mu(r)}\left(1 +  \frac{\ell_1}{r^2}\right)}}g_\Phi(r,\mu; \delta) + E^2\left(E^2 - e^{2\mu(r)}\left(1 +  \frac{\ell_1}{r^2}\right)\right)^{\frac{1}{2}}D_\mu g_\Phi(r,\mu, E; \delta)[\tmu]dE \right). 
\end{aligned}
\end{equation}
and 
\begin{equation}
\label{D3H:Phi}
\begin{aligned}
D_\mu H_\Phi(r, \mu; \delta)[\tmu] &=  2\pi e^{-4\mu(r)}\left( -4\tmu(r) \int_{e^{\mu(r)}\sqrt{1 + \frac{\ell_1}{r^2}}}^{E_2}\;E^2\left(E^2 - e^{2\mu(r)}\left(1 +  \frac{\ell_1}{r^2}\right)\right)^{\frac{3}{2}}h_\Phi(r,\mu, E; \delta)dE\right.   \\
&\left. +  \int_{e^{\mu(r)}\sqrt{1 + \frac{\ell_1}{r^2}}}^{E_2}\;-3E^2\tmu(r)e^{2\mu(r)}\left(1 + \frac{\ell_1}{r^2} \right)\sqrt{E^2 - e^{2\mu(r)}\left(1 +  \frac{\ell_1}{r^2}\right)}h_\Phi(r,\mu, E; \delta) \right. \\
&\left. + E^2\left(E^2 - e^{2\mu(r)}\left(1 +  \frac{\ell_1}{r^2}\right)\right)^{\frac{1}{2}}D_\mu h_\Phi(r,\mu, E; \delta)[\tmu]dE \right). 
\end{aligned}
\end{equation}
\end{enumerate}
\end{proof}
\subsection{The condition \eqref{cond::m} is satisfied}
\label{Condition::m::r}
In this section, we show that we can choose $\delta_0$ even smaller so that the condition \eqref{cond::m} is satisfied. More precisely, we state the following lemma
\begin{lemma}
\label{condition::}
There exists $\delta_0\in]0, \tilde\delta_0] $ such that $\displaystyle\forall(\mu; \delta)\in\mathcal U(\delta_0), \; \forall r\in I$
\begin{equation}
2m(\mu; \delta)(r) < r. 
\end{equation}
Moreover, there exists $C>0$ such that $\forall r\in\Ibarre$, 
\begin{equation*}
\left(1 - \frac{2m(\mu; \delta)(r)}{r}\right)^{-1} < C. 
\end{equation*}
\end{lemma}

\begin{proof}
\begin{enumerate}
\item First, we recall
\begin{equation*}
2m(\mu; \delta)(r) := 2M + 8\pi\int_{2M+\rho}^r\, s^2G_\Phi(s, \mu; \delta)\;ds
\end{equation*}
where $G_\Phi(s, \mu; \delta)$ is given by 
\begin{equation*}
\displaystyle G_\Phi(r, \mu; \delta) = 2\pi e^{-4{\mu(r)}}\int_{e^{\mu(r)}\sqrt{1 + \frac{\ell_1}{r^2}}}^{E_2}\;E^2\left(E^2 - e^{2{\mu(r)}}\left(1 +  \frac{\ell_1}{r^2}\right)\right)^{\frac{1}{2}}g_\Phi(r, \mu, E;\delta)dE. 
\end{equation*}
\item If $\displaystyle 2M + \rho < r \leq R_{min}(\mu)$, then 
\begin{equation*}
2m(\mu; \delta)(r) = 2M < r.
\end{equation*}
Hence, the condition \eqref{cond::m} is always satisfied. Besides, $\displaystyle \forall r\in\left]2M+\rho, R_{min}(\mu) \right[$, we have
\begin{equation*}
\left(1 - \frac{2m(\mu; \delta)(r)}{r}\right)^{-1} < 1 + \frac{2M}{\rho}. 
\end{equation*}

\item Now, let $r\geq R_{min}(\mu)$. We claim that  $\exists\, C>0$, independent of $(\mu, \delta)$ such that 
\begin{equation*}
||r\mapsto 8\pi G_\Phi(r, \mu; \delta)||_{C^1(\Ibarre)}\leq C\delta_0. 
\end{equation*}
In fact, for a fixed $r\in I$ we write
\begin{align*}
G_\Phi(r, \mu; \delta) &= G_\Phi(r, \mu^{Sch}; 0) + D_\mu G_\Phi(r, \mu^{Sch}; 0)[\mu - \mu^{Sch}]  + \delta\partial_\delta G_\Phi(r, \mu^{Sch}; 0) + O\left(||\mu - \mu^{Sch}||^2_{C^1(\Ibarre)} + \delta^2\right) \\
&= \delta\partial_\delta G_\Phi(r, \mu^{Sch}; 0) + O(\delta_0),
\end{align*}
where we used 
\begin{equation*}
G_\Phi(r, \mu^{Sch}; 0) = D_\mu G_\Phi(r, \mu^{Sch}; 0)[\mu - \mu^{Sch}] = 0. 
\end{equation*}
Now we have 
\begin{align*}
&\partial_\delta G_\Phi(r, \mu^{Sch}; 0) =  \\
&2\pi\left(1 - \frac{2M}{r}\right)^{-2}\int_{\sqrt{\left(1 - \frac{2M}{r}\right)\left(1 +  \frac{\ell_1}{r^2}\right)}}^{E_2}\,E^2\left(E^2 - \left(1 - \frac{2M}{r}\right)\left(1 +  \frac{\ell_1}{r^2}\right)\right)^{\frac{1}{2}}\partial_\delta g_\Phi(r, \mu^{Sch}, E;0)\,dE. 
\end{align*}
Since $r\in\Ibarre$ and by the the $(\Phi_2)$ assumption and the definition of $\Chi_\eta$, there exists $\displaystyle\, C>0$ independent of $(r, \mu; \delta)$ such that
\begin{equation*}
\forall E \in [E_1, E_2], \forall r \in I\; \; \left|\partial_\delta g_\Phi(r, \mu^{Sch}, E;0)\right| \leq C. 
\end{equation*}
Therefore, 
\begin{align*}
&\left|\partial_\delta G_\Phi(r, \mu^{Sch}; 0) \right| \\
&\leq  2C\delta\pi\left( 1 - \frac{2M}{2M+\rho}\right)^{-2}\int_{\sqrt{\left(1 - \frac{2M}{r}\right)\left(1 +  \frac{\ell_1}{r^2}\right)}}^{E_2}\,E^2\left(E^2 - \left(1 - \frac{2M}{2M + \rho}\right)\left(1 +  \frac{\ell_1}{(2M+\rho)^2}\right)\right)^{\frac{1}{2}}\,dE + O(\delta_0) \\
&\leq C\delta\int_0^{E_2}E^3\,dE + O(\delta_0) \\
&\leq C\delta_0.  
\end{align*}
This yields the result. 
\\We have, by \eqref{Rmin::cond}, $\forall \mu \in B(\mu^{Sch}, \delta_0)$, 
\begin{equation*}
R_{min}(\mu) > 4M. 
\end{equation*}
Therefore,
\begin{align*}
2m(\mu; \delta)(r) &\leq \frac{R_{min}(\mu)}{2} + \frac{C\delta_0}{3}\left( r^3 - R_{min}(\mu)^3 \right)^2 \\
&\leq \frac{R_{min}(\mu)}{2} + \frac{CR^2\delta_0}{3}r \\
&\leq\left(\frac{1}{2} +  \frac{CR^2\delta_0}{3}\right)r. 
\end{align*}
It remains to update $\delta_0$ so that 
\begin{equation*}
\frac{1}{2} +  \frac{CR^2\delta_0}{3} < 1. 
\end{equation*}
We take for example
\begin{equation*}
\delta_0 = \min\left(\delta_0, \frac{3}{4CR^2} \right). 
\end{equation*} 
Therefore,  $\displaystyle \forall r\geq R_{min}(\mu),$ we have
\begin{equation*}
\left(1 - \frac{2m(\mu; \delta)(r)}{r}\right)^{-1} < 4. 
\end{equation*}

\end{enumerate}
This concludes the proof.
\end{proof}

\subsection{Solving for $\mu$}
\label{Solving::mu}

We check in a number of steps that the mapping $G$ defined by \eqref{G::mu} satisfies the conditions for applying the implicit function theorem: 
\begin{enumerate}
\item First, we need to check that $G$ is well defined on $\mathcal U(\delta_0)$. 
\item It is clear that $(\mu^{Sch}; 0)$ is a zero for $G$. 
\item Next, we need to check that $G$ is continuously Fréchet differentiable on $\mathcal U(\delta_0)$. 
\item Finally, we have to show that the partial Fréchet derivative with respect to the first variable $\mu$ at the point $(\mu^{Sch};0)$:
\begin{align*}
\mathcal L : C^1(\overline I)&\to C^1(\overline I) \\
\;\mu &\mapsto DG_{(\mu^{Sch};0)}\left(\mu; 0\right)
\end{align*}  
is invertible.
\end{enumerate}
Provided the above facts hold, we can now apply Theorem \ref{IFT} to the mapping $G:\mathcal U\to \mathcal X$ to obtain:
\begin{theoreme}
\label{res::version3}
There exists $\delta_1, \delta_2\in]0, \delta_0[$ and a unique differentiable solution map 
\begin{equation*}
\mu:[0, \delta_1[\to B\left(\mu^{Sch}, \delta_2 \right)\subset C^1(\overline I), 
\end{equation*}
such that
$\displaystyle S(0) = \mu^{Sch}$ and 
\begin{equation*}
G(\mu(\delta);\delta) = 0, \quad\quad\forall \delta\in [0, \delta_1[. 
\end{equation*} 
\end{theoreme}
The remainder of this section is devoted to the proofs of steps $1$ to $4$. 
\subsubsection{$G$ is well defined on $\mathcal U$}
Let $\displaystyle (\mu; \delta)\in\mathcal U(\delta_0)$. Recall the definition of $G$
\begin{equation*}
G(\mu; \delta)(r) :=  \mu(r) -  \left( \mu_0 + \int_{2M+\rho}^{r}\frac{1}{1 - \frac{2m(\mu; \delta)(s)}{s}}\left(4\pi s H_{\Phi}(s, \mu; \delta) + \frac{1}{s^2}m(\mu; \delta)(s)\right)\,ds \right). 
\end{equation*}
By  Lemma \ref{condition::}, there exists $C>0$ such that 
\begin{equation}
\label{cond::n}
\forall r\in I,\;\; \left( 1 - \frac{2m(\mu; \delta)(r)}{r}  \right)^{-1} < C. 
\end{equation}
Besides, $\displaystyle \forall r\in I$,
\begin{align*}
\int_{2M+\rho}^{r}\left(4\pi s H_{\Phi}(s, \mu; \delta) + \frac{1}{s^2}m(\mu; \delta)(s)\right)\frac{ds}{1 - \frac{2m(\mu; \delta)(s)}{s}}  &\leq \int_{2M+\rho}^{R_{max}(\mu)}\left(4\pi s H_{\Phi}(s, \mu; \delta) + \frac{1}{s^2}m(\mu; \delta)(s)\right)\frac{ds}{1 - \frac{2m(\mu; \delta)(s)}{s}} .
\end{align*}
The integral  of the right hand side is finite. Therefore, $G(\mu; \delta_0)$ is well defined on $I$. Moreover, by Proposition \ref{regularity::G::H} and lemma \ref{condition::} we have
\begin{itemize}
\item $\displaystyle r\mapsto H_\Phi(r, \mu; \delta)$ is $C^1$ on $I$, 
\item $\displaystyle r\mapsto m(\mu; \delta)(r)$ is $C^1$ on $I$,
\end{itemize}
Hence, $G(\mu; \delta)$ is well defined and it is $C^1$ on $I$. Moreover, it easy to control its $C^1$ norm thanks to \eqref{cond::n} and the compact support of $r\mapsto H_\Phi(r, \mu; \delta)$ and $r\mapsto G_\Phi(r, \mu(r), \mu; \delta)$. 

\subsubsection{$G$ is continuously Fréchet differentiable on $\mathcal U$}
Let $(\mu; \delta)\in\mathcal U(\delta_0)$. First, we show the differentiability with respect to $\mu$. In this context, we drop the dependence on $\delta$ in order to lighten the expressions. Our first claim is that $G$ has the Fréchet derivative with respect to $\mu$ given by $\displaystyle \forall\tmu\in C^1(\Ibarre)\,, \,\forall r\in I,$
\begin{equation}
\label{DG:mu}
\begin{aligned}
\;\; &D^1G(\mu)[\tmu](r) := \tmu(r) - \int_{2M+\rho}^{r}\frac{2}{s}\frac{1}{\displaystyle \left(1-\frac{2m(\mu)(s)}{s}\right)^2}Dm(\mu)[\tilde\mu](s)\left( 4\pi sH_\Phi(s, \mu) + \frac{1}{s^2} m(\mu)(s)\right)  \\
&+ \frac{1}{\displaystyle 1 - \frac{2 m(\mu)(s)}{s}}\left( 4\pi s D_\mu H_\Phi(s, \mu)[\tmu] + \frac{1}{s^2}D m(\mu)[\tilde\mu](s) \right)ds, 
\end{aligned}
\end{equation} 
where $Dm(\mu)$ is the Fréchet derivative of $m$ with respect to $\mu$, given by 
\begin{equation}
\label{Dm::mu}
Dm(\mu)[\tmu](r) = 4\pi\int_{2M+\rho}^{r}\, s^2D_\mu G_\Phi(s, \mu)[\tmu] ds.
\end{equation}
We need to prove
\begin{equation*}
\lim_{||\tilde\mu||_{C^1(\Ibarre)}\to 0}\; \frac{\left|\left|G(\mu+\tilde\mu) - G(\mu) -  D^1G(\mu)[\tmu]\right|\right|_{C^1(\Ibarre)}}{||\tilde\mu||_{C^1(\Ibarre)}} = 0. 
\end{equation*}

We will only prove in details  
\begin{equation*}
\lim_{||\tilde\mu||_{C^1(\Ibarre)}\to 0}\; \frac{\left|\left|m(\mu+\tilde\mu) - m(\mu) -  Dm(\mu)[\tmu]\right|\right|_{C^1(\Ibarre)}}{||\tilde\mu||_{C^1(\Ibarre)}} = 0. 
\end{equation*}
The remaining terms are treated in the same way thanks to the estimate \eqref{cond::n}, the compact support in $r$ of the matter terms as well as their regularity. 
\\ Let $\displaystyle r\in I$. We compute 
\begin{align*}
& m(\mu + \tilde\mu)(r) - m(\mu)(r) -   Dm(\mu)[\tilde\mu](r)  = 4\pi \int_{2M+\rho}^r\;  s^2\left(  G_{\Phi}(s, \mu + \tilde\mu) - G_{\Phi}(s, \mu) - D_\mu G_{\Phi}(s, \mu)[\tilde\mu]  \right)\;ds.
\end{align*}
Since $G_\Phi$ is  continuously Fréchet differentiable with respect to $\mu$, we write in a neighbourhood of the point $\displaystyle \mu$ with a fixed $s$
\begin{align*}
G_{\Phi}(s, \mu + \tilde\mu) &=  G_{\Phi}(s, \tilde\mu)  + D_\mu G_\Phi(s, \mu)[\tmu] + O\left(||\tmu||^2_{C^1(\Ibarre)}\right) .
\end{align*}
Therefore, 
\begin{equation*}
\left| m(\mu + \tmu)(r) - m(\mu)(r) - Dm(\mu)[\tilde\mu](r) \right| \lesssim ||\tmu||^2_{C^1(\Ibarre)},
\end{equation*}
so that 
\begin{equation*}
\lim_{\displaystyle||\tmu||_{C^1(\Ibarre)}\to 0}\; \frac{\displaystyle||m(\mu + \tmu) - m(\mu) -Dm(\mu)[\tmu] ||_\infty}{\displaystyle||\tmu||_{C^1(\Ibarre)}}= 0.
\end{equation*}
It remains to control the $C^1$ norm of $m(\mu + \tmu) - m(\mu) -   Dm(\mu)[\tmu]$. It is clear that the latter is $C^1$ on $I$ and the derivative is given by $\forall r\in I$  :
\begin{align*}
m(\mu + \tmu)'(r) - m(\mu)'(r) -   Dm(\mu)[\tmu]'(r)   &=  4\pi r^2\left(G_{\Phi}(r, \mu + \tmu)  -  G_{\Phi}(r, \mu) - D_\mu G_\Phi(r, \mu)[\tmu]) \right)   \\
&= O\left(||\tmu||^2_{C^1(\Ibarre)}\right), 
\end{align*}
which implies 
\begin{equation*}
\lim_{\displaystyle||\tmu||_{C^1(\Ibarre)}\to 0}\; \frac{\displaystyle||m(\mu + \tmu)' - m(\mu)' -Dm(\mu)[\tmu]' ||_\infty}{\displaystyle||\tmu||_{C^1(\Ibarre)}}= 0.
\end{equation*}
To conclude, we show that the Fréchet derivative is continuous on $\displaystyle B(\mu^{Sch}, \delta_0)$. Let $\displaystyle \mu_1, \mu_2\in B(\mu^{Sch}, \delta_0)$ and let $\tmu\in C^1(\Ibarre) \; ||\tmu||_{C^1(\Ibarre)}\leq 1$, we compute $\displaystyle\forall r\in I$
\begin{align*}
Dm(\mu_1)(\tmu)(r) - Dm(\mu_2)(\tmu)(r) &=  4\pi\int^r_{2M+\rho}\;s^2\left(D_\mu G_\Phi(s, \mu_1)[\tmu]) - D_\mu G_\Phi(s, \mu_2)[\tmu])\right)\, ds.
\end{align*}
We use the regularity results of Proposition \ref{regularity::G::H} and the estimate $||\tmu||_{C^1(\Ibarre)}\leq 1$ to estimate each term by $\displaystyle ||\mu_1 - \mu_2||_{C^1(\Ibarre)}$. Similarly, we control $\displaystyle ||Dm(\mu_1)(\tmu)' - Dm(\mu_2)(\tmu)'||_\infty$. By taking the supremum on $||\tmu||_{C^1(\Ibarre)}$, we obtain 
\begin{equation*}
||Dm(\mu_1) - Dm(\mu_2)||_{\mathcal L(C^1(\Ibarre), C^1(\Ibarre))} \lesssim ||\mu_1 - \mu_2||_{C^1(\Ibarre)}. 
\end{equation*}
Finally, when $\mu_1\to\mu_2 $ in $C^1(\Ibarre)$ we obtain $\displaystyle Dm(\mu_1)\to Dm(\mu_2) $ in $\mathcal L(C^1(\Ibarre), C^1(\Ibarre))$. This proves the continuity. 
\\ 
\\ We check now the differentiability with respect to $\delta$. We claim that
\begin{equation}
\begin{aligned}
D_\delta G(\delta)(r) &= - \int_{2M+\rho}^{r}\frac{2}{s}\frac{1}{\displaystyle \left(1-\frac{2m(\mu)(s)}{s}\right)^2}m'(\delta)(s)\left( 4\pi sH_\Phi(s, \delta) + \frac{1}{s^2} m(\mu)(s)\right)  \\
&+ \frac{1}{\displaystyle 1 - \frac{2 m(\mu)(s)}{s}}\left( 4\pi s\partial_\delta H_\Phi(s; \delta) + \frac{1}{s^2}m'(\delta)(s) \right)ds, 
\end{aligned}
\end{equation}
where 
\begin{equation*}
m'(\delta)(s) = 4\pi\int_{2M+\rho}^r\,s^2\partial_\delta G_\Phi(s, \mu; \delta)\,ds, 
\end{equation*}
and 
\begin{equation*}
\partial_\delta H_\Phi(s; \delta) = \partial_\delta H_\Phi(s, \mu; \delta). 
\end{equation*}
Again, we dropped the dependance on $\mu$ in order to lighten the expressions. By regularity assumptions on $\Phi$ with respect to the third variable, the compact support of the matter terms and by the dominated convergence theorem, G is $C^1$ with respect to the parameter $\delta$. 
\\ 
\\ 
\\ 
We check the third step: we evaluate the derivative of $G$ with respect to $\mu$ at the point $\displaystyle (\mu^{Sch}, 0)$. For this we use Proposition \ref{regularity::G::H}  and the fact that  $\displaystyle \Phi(\cdot, \cdot ; 0) = \partial_\ell\Phi(\cdot, \cdot; 0) =0\; $to obtain 
\begin{equation*}
D_\mu G_\Phi(s, \mu^{Sch} ; 0) =  D_\mu H_\Phi(s, \mu^{Sch}; 0) = 0.
\end{equation*}
Therefore, $\displaystyle D_\mu G(\mu^{Sch}; 0)$ is reduced to the identity, which is invertible. Since all the assumptions are satisfied, we apply Theorem \ref{IFT} to obtain Theorem \ref{res::version3}. 

\subsection{Conclusions}
In this section, we deduce from Theorem \ref{res::version3} the remaining metric component and the expression of the matter terms. More precisely, we set $\displaystyle\forall \delta\in[0, \delta_1[$
\begin{equation}
\label{lambda:delta}
\lambda^\delta(r) := -\frac{1}{2}\log\left(1 - \frac{2m(\mu(\delta); \delta)}{r} \right), \quad\quad\forall r\in I. 
\end{equation}
$\lambda^\delta$ is well defined since 
\begin{equation*}
\forall r\in I, \quad\quad r > 2m(\mu(\delta); \delta). 
\end{equation*}
As for the matter field, we set $\displaystyle\forall (x^\alpha, v^a)\in\mathcal O\times\mathbb R^3$
\begin{equation}
f^\delta(x^\alpha, v^a) := \Phi(E_\delta, \ell; \delta)\Psi_\eta\left(r - r_1(\mu(\delta), E^\delta, \ell)\right), 
\end{equation}
where
\begin{equation*}
E^\delta = e^{2\mu_\delta(r)}\sqrt{1 + (e^{\lambda^\delta(r)}v^r)^2 + (rv^\theta)^2 + (r\sin\theta v^\phi)^2}
\end{equation*}
and $\ell$ is given by \eqref{ang::momentum}.
\\Such $f^\delta$ is a solution to the Vlasov equation \eqref{Liouville::} since $\Phi$ is a solution and $\Psi_\eta$ is constant on the support of $\Phi$.   
Furthermore, we set the {\it {energy density}}  to be
\begin{equation}
\rho_\delta(r) := G_\Phi(r, \mu(\delta);\delta),
\end{equation}
We define the radii of the matter shell to be
\begin{equation}
\label{Rmin:delta}
R_{min}^\delta := R_{min}(\mu(\delta))
\end{equation}
and
\begin{equation}
\label{Rmax:delta}
R_{max}^\delta := R_{max}(\mu(\delta)),
\end{equation}
where $R_{min}$ and $R_{max}$ are defined by \eqref{Rmin} and \eqref{Rmax}.  Finally, we define the {\it{total mass}} by
\begin{equation}
\label{total::mass}
M^\delta := M + 4\pi\int_{R_{min}^\delta }^{R_{max}^\delta}r^2\rho_\delta(r)\,dr. 
\end{equation}
It is easy to see that $\mu(\delta)$ and  $\lambda^\delta$ are actually $C^2$ on $I$. In fact, since $\mu(\delta)$ is $C^1$ on $I$, $\displaystyle m(\mu(\delta); \delta)$ given by \eqref{mass} is $C^2$ on $I$ and $r\mapsto H(r,  \mu(\delta); \delta)$ is $C^1$ on $I$ by Proposition \ref{regularity::G::H}. Moreover, 
\begin{equation*}
G(\mu(\delta); \delta) = 0.
\end{equation*}
That is $\mu(\delta)$ is  implicitly given by 
\begin{equation*}
 \mu(\delta)(r) =  \left( \mu_0 + \int_{2M+\rho}^{r}\frac{1}{1 - \frac{2m(\mu(\delta); \delta)(s)}{s}}\left(4\pi s H_{\Phi}(s,  \mu(\delta); \delta) + \frac{1}{s^2}m(\mu(\delta); \delta)(s)\right)\,ds \right). 
\end{equation*}
From the above equation it is clear that $\mu(\delta)$ is $C^2$ on $I$. Therefore, $\lambda^\delta$ is also $C^2$ on $I$ and $f^\delta$ is also $C^2$ with respect to $r$. 
\\It remains to extend the solution on $]2M, \infty[$. It suffices to extend $\displaystyle (\mu, \lambda)$ by $\displaystyle (\mu^{Sch}, \lambda^{Sch})$ of mass $M$ on the domain $\displaystyle ]2M, 2M+\rho]$, and by $\displaystyle (\mu^{Sch}, \lambda^{Sch})$ of mass $M^\delta$ on the domain $[R, \infty[$. We recall that the matter field vanishes in these regions thanks to its compact support. Thus, $(\mu, \lambda)$ are still $C^2$ on $]2M, \infty[$. This ends the proof of Theorem \ref{main::result}.

\clearpage 
\renewcommand*\appendixpagename{Appendix}
\renewcommand*\appendixtocname{Appendix}
\appendix

\addappheadtotoc
\section{Study of the geodesic motion in the exterior of Scwharzschild spacetime}
\label{Schwarzschild}
We present a detailed study of the geodesic motion in the exterior region of a fixed Schwarzschild spacetime. We will classify the geodesics based on the study of the effective energy potential.  
Such a classification is of course classical and we refer to \cite[Chapter 3]{Chandrasekhar:1985kt} or \cite[Chapter 33]{misner2017gravitation} for more details. 
%\subsection{Geodesic motion in the Schwarzschild spacetime}
In this section, we prove Proposition \ref{Propo1}. 
\begin{itemize}
\item First, note that $\Esch$ is a cubic function in $\displaystyle \frac{1}{r}$. Its derivative is given by 
\begin{equation*}
\forall r> 2M\; , \quad  {\Esch}'(r) = \frac{2}{r^4}\left(Mr^2 -\ell r + 3M\ell\right). 
\end{equation*}
Three cases are possible :
\begin{enumerate}
\item $\displaystyle {\Esch}'$ has two distinct roots $\displaystyle r_{max}^{Sch}<r_{min}^{Sch}$, where $r_{max}^{Sch}$ and $r_{min}^{Sch}$ correspond respectively to the maximiser and the minimiser of $\displaystyle {\Esch}$.  They are given by
\begin{equation*}
%\label{r:max*}
r_{max}^{Sch}(\ell) = \frac{\ell}{2M}\left( 1 - \sqrt{1-\frac{12M^2}{\ell}}\right),
\end{equation*} 
\begin{equation*}
%\label{r:min}
r_{min}^{Sch}(\ell) = \frac{\ell}{2M}\left( 1 + \sqrt{1-\frac{12M^2}{\ell}}\right). 
\end{equation*}
The extremums of ${\Esch}$ are given by 
\begin{equation}
\label{E:Sch:min}
E^{min}(\ell) = \Esch(r_{min}^{Sch}) = \frac{8}{9}+\frac{\ell-12M^2}{9Mr_{min}^{Sch}(\ell)}, 
\end{equation}
\begin{equation}
\label{E:Sch:max}
E^{max}(\ell) = \Esch(r_{max}^{Sch}) = \frac{8}{9}+\frac{\ell-12M^2}{9Mr_{max}^{Sch}(\ell)}.
\end{equation}
In fact, we compute 
\begin{equation*}
1 - \frac{2M}{r_{min}^{Sch}(\ell)} = 1 - \frac{4M^2}{\ell\left(1 - \sqrt{1 - \frac{12M^2}{\ell}} \right)} =  \frac{2}{3} - \frac{1}{3}\sqrt{1 - \frac{12M^2}{\ell}},
\end{equation*}
and
\begin{equation*}
1 + \frac{\ell}{r_{min}^{Sch}(\ell)^2} = 1 + \frac{\ell}{36M^2}\left( 1 + \sqrt{1 - \frac{12M^2}{\ell}}\right)^2.
\end{equation*}
Therefore, 
\begin{align*}
E_\ell^{Sch}(r_{min}^{Sch}(\ell)) &= \left(1 - \frac{2M}{r_{min}^{Sch}(\ell)}\right)\left( 1 + \frac{\ell}{r_{min}^{Sch}(\ell)^2} \right) \\
&= \left(\frac{2}{3} + \frac{\ell}{18M^2} + \frac{\ell}{18M^2}\sqrt{1 - \frac{12M^2}{\ell}}\right)\left(\frac{2}{3} - \frac{1}{3}\sqrt{1 - \frac{12M^2}{\ell}}\right) \\
&=\frac{8}{9}+\frac{\ell-12M^2}{9Mr_{min}^{Sch}(\ell)},
\end{align*}
where the last expression is obtained by straightforward computations. We do the same thing for $E^{max}(\ell)$. 
This case occurs if and only if $\displaystyle \ell>12M^2$. 
\item $\displaystyle {\Esch}'$ has one double root at $r_c = 6M$. The extremum of $\Esch$ is given by 
\begin{equation*}
E_\ell^c = \frac{8}{9}. 
\end{equation*}
This case occur if and only if $\ell = 12M^2$. 
\item $\displaystyle {\Esch}'$ has no real roots. Then, $\displaystyle {\Esch}$ is monotonically increasing from $0$ to $+\infty$. 
This case occurs if and only if $\ell < 12M^2$. 
\end{enumerate}
We refer to Figure \ref{Fig::2} in the introduction for  the shape of the potential energy in the three cases. 
\item Note that by the mass shell condition \eqref{r::pr}, we have 
\begin{equation*}
E^2 \geq E_\ell^{Sch}(r)
\end{equation*}
for any timelike geodesic moving in the exterior region.  In particular, 
\begin{equation*}
E^2 \geq E^{min}(\ell) \geq \frac{8}{9}.  
\end{equation*}
Therefore, we obtain a lower bound on $E$: 
\begin{equation*}
E \geq \sqrt{\frac{8}{9}}. 
\end{equation*}
\item Now, we claim that the trajectory is a circle of radius $r_0^{Sch} > 2M$ if and only if 
\begin{equation}
\label{circ::motion}
\Esch(r_0^{Sch}) = E^2 \quad\text{and}\quad {\Esch}'(r_0^{Sch}) = 0.  
\end{equation}
Indeed, if the motion is circular of radius $r_0^{Sch}$, then $\displaystyle \forall \tau\in\mathbb R\;:\;  r(\tau) = r_0^{Sch}$. Thus,
\begin{equation*}
\forall\tau\in\mathbb R \;,\; w(\tau) = \dot r(\tau) = 0 \quad\text{and}\quad  \dot w(\tau) = 0. 
\end{equation*}
Besides, it is a solution to the system  \eqref{r:motion}-\eqref{pr:motion}. Therefore 
\begin{equation*}
\forall\tau\in\mathbb R \;,\; {\Esch}'(r_0^{Sch}) = {\Esch}'(r(\tau)) = 0.
\end{equation*}
By \eqref{r::pr}, we have 
\begin{equation*}
\forall\tau\in\mathbb R \;,\; w(\tau)^2 + {\Esch}(r(\tau)) = E^2. 
\end{equation*}
In particular, 
\begin{equation*}
{\Esch}(r_0^{Sch}) = E^2. 
\end{equation*}
Now, let us suppose that there exists $r_0^{Sch}>2M$ such that 
\begin{equation*}
\Esch(r_0^{Sch}) = E^2 \quad\text{and}\quad {\Esch}'(r_0^{Sch}) = 0. 
\end{equation*}
By the above assumption, we have $w = 0$ and ${\Esch}'(r_0^{Sch}) = 0$ so that the point $(r_0^{Sch}, 0)$ is a stationary point $\forall \ell\geq 12 M^2$. Furthermore,
\begin{enumerate}
\item $r_0^{Sch}= 6 $, if $\ell = 12M^2$, 
\item $r_0^{Sch}\in\left\{r_{min}^{Sch}(\ell), r_{max}^{Sch}(\ell)\right\}$, if $\ell > 12M^2$. 
\end{enumerate}
The circular orbits are thus characterised by  \eqref{circ::motion}. 
\item We consider now the case $\ell>12M^2$ and the equation 
\begin{equation}
\label{cond::roots}
E^2 = \Esch(r).
\end{equation}
Let $\displaystyle (\ell, E)\in\left]12M^2, \infty\right[\times\left]\frac{8}{9}, \infty\right[$. Four cases may occur:
\begin{enumerate}
\item If $\displaystyle E^2 = E_\ell^{min}$ or  $\displaystyle E^2 = E_\ell^{max}$, then $r_{min}^{Sch}$ or $r_{max}^{Sch}$ satisfy \eqref{circ::motion}, so that they are double roots. Besides, we note that $\ell = \ell_{ub}$ where
\begin{equation}
\label{l:ub}
\ell_{ub}(E) := \frac{12M^2}{1-4\alpha-8\alpha^2-8\alpha\sqrt{\alpha^2+\alpha}}  \quad\quad \alpha := \frac{9}{8}E^2 - 1
\end{equation}
satisfies $E^2 = E_\ell^{min}$. In fact, we solve the equation below for $\ell$
\begin{align*}
\frac{8}{9} + \frac{\ell - 12M^2}{9\frac{\ell}{2}\left( 1 + \sqrt{1-\frac{12M^2}{\ell}}\right)} = E^2.
\end{align*}
We make the following change of variables
\begin{align*}
\alpha := \frac{9}{8}E^2 - 1\quad\text{and}\quad X := \sqrt{\frac{\ell - 12M^2}{\ell}}.
\end{align*}
The equation becomes
\begin{equation*}
\frac{X^2}{4\left(1 + X \right)}= \alpha,
\end{equation*} 
which is easily solvable for $X$. We can then obtain $\ell_{ub}(E)$. Similarly, we obtain $\ell_{lb}$ defined by 
\begin{equation}
\label{l:lb}
\ell_{lb}(E) := \frac{12M^2}{1-4\alpha-8\alpha^2+8\alpha\sqrt{\alpha^2+\alpha}}, 
\end{equation}
which solves the equation   $E^2 = E_\ell^{max}$. 
\item  If $\displaystyle E^2 > E_\ell^{max}$, then two cases occur
\begin{enumerate}
\item $E^2 < 1$. The equation  \eqref{cond::roots} has one simple root $r_2^{Sch}(E, \ell)> r_{min}^{Sch}(\ell)$
\item $E^2 \geq 1$. The equation \eqref{cond::roots} has no positive roots. The trajectories in this case are similar to the trajectories in case 3 (where $\ell <12M^2$). 
\end{enumerate}

\item $\displaystyle E^2 \in \left]E_\ell^{min}, E_\ell^{max}\right[$. Again, two cases occur 
\begin{enumerate}
\item $E^2 < 1$. Then the equation \eqref{cond::roots} admits three simple positive roots $r_i^{Sch}(E, \ell)$
\begin{equation}
\label{simple::roots}
r_0^{Sch}(E, \ell) < r_{max}^{Sch}(\ell) < r_1^{Sch}(E, \ell) < r_{min}^{Sch}(\ell) < r_2^{Sch}(E, \ell). 
\end{equation}
\item $E^2 \geq 1$. The equation \eqref{cond::roots} admits two simple positive roots $r_i^{Sch}(E, \ell)$
\begin{equation}
\label{simple:roots}
r_0^{Sch}(E, \ell) < r_{max}^{Sch}(\ell) < r_1^{Sch}(E, \ell) < r_{min}^{Sch}(\ell) . 
\end{equation}
\end{enumerate}
\end{enumerate}

\item We consider now the case $\ell \leq 12M^2$. Three cases may occur:

\begin{enumerate}
\item If $\displaystyle E^2 = \frac{8}{9}$, then the equation \eqref{cond::roots} has one triple positive root  $\displaystyle r_c = 6M$.  
\item If $\displaystyle \frac{8}{9}<E^2<1$, then the equation \eqref{cond::roots} has one simple positive root $r_1^{Sch}(E, \ell)$. 
\item If $E^2 \geq 1$, then the equation \eqref{cond::roots} no positive roots.  
\end{enumerate}

\item Based on the above cases, we define the following parameters sets 

\begin{equation}
\begin{aligned}
\mathcal  A_{circ} &:= \left\{(\sqrt{\frac{8}{9}}, 12M^2)\right\}\bigcup\left\{ (E, \ell)\in\ELrange\;:\; E^2 < 1,\quad  \ell = \ell_{ub}(E) \right\} \\
&\bigcup\left\{ (E, \ell)\in\ELrange\;:\;  \ell = \ell_{lb}(E) \right\},
\end{aligned}
\end{equation}

\begin{equation}
%\label{A::bound}
\mathcal  A_{bound} := \left\{ (E, \ell)\in\ELrange \;:\; E^2 <1 ,\quad \ell_{lb}(E) <\ell< \ell_{ub}(E) \right\},
\end{equation}

\begin{equation}
\mathcal  A_{unbound} := \left\{ (E, \ell)\in\ELrange \;:\; E^2 \geq 1 ,\quad \ell> \ell_{lb}(E) \right\}, 
\end{equation}

\begin{equation}
\begin{aligned}
\mathcal  A_{abs} &:= \left\{ (E, \ell)\in\ELrange \;:\;  E^2 < 1, \quad \ell< \ell_{lb}(E) \right\}\bigcup  \left\{ (E, \ell)\in\left]\sqrt{\frac{8}{9}}, 1 \right[\times \left]0, 12M^2 \right[ \right\} \\
&\bigcup \left\{ E\in\left] \sqrt{\frac{8}{9}}, 1\right[ \; ,\; \ell = 12M^2\right\}.
\end{aligned}
\end{equation}
\item Now, we determine the nature of orbits (circular, bounded, unbounded, "absorbed by the black hole") in terms of the parameters $(E, \ell)$ as well as the initial position and velocity. Let $\ell\in[0, \infty[$, $\displaystyle \tilde r\in]2M, \infty[$ and $\displaystyle \tilde w\in\mathbb R$. We compute
\begin{equation*}
E = \sqrt{\Esch(\tilde r)  + \tilde w^2}.
\end{equation*}
\begin{enumerate}
\item If $(E, \ell)\in\Abound$, then there exists $r_i^{Sch} := r_i^{Sch}(E, \ell), i\in\left\{0, 1, 2\right\}$ solutions of \eqref{cond::roots} and satisfying \eqref{simple::roots}. Now recall that 
\begin{equation}
\label{cond::motion}
\Esch(r) \leq E^2, \quad\forall r>2M.
\end{equation}
This implies that $\tilde r$ must lie in the region $\displaystyle \left]2M, r_0^{Sch}\right]\cup[r_1^{Sch}, r_2^{Sch}]$. 
Two cases are possible 
\begin{itemize}
\item either the geodesic starts at $r_0^{Sch}$ and reaches the horizon $r= 2M$ in a finite proper time,
\item or the geodesic is trapped between $r_1^{Sch}$ and $r_2^{Sch}$. 
\end{itemize}
\item If $(E, \ell)\in \mathcal A_{unbound}$, then there exists $r_i^{Sch} := r_i^{Sch}(E, \ell), i\in\left\{0, 1, \right\}$ solutions of \eqref{cond::roots} and satisfying \eqref{simple:roots}. By \eqref{cond::motion}, $\tilde r$ must lie in the region $]2M, r_0^{Sch}]\cup[r_1^{Sch}, \infty[$. Therefore, two cases are possible
\begin{itemize}
\item either the geodesic starts at $r_0^{Sch}$ and reaches the horizon $r= 2M$ in a finite proper time,
\item or the geodesic stars from infinity, hits the potential barrier at $r_1^{Sch}$ and comes back to infinity. 
\end{itemize}
\item If $(E, \ell)\in \mathcal A_{abs}$, then the equation \eqref{cond::roots}  has at most one positive root $r_0^{Sch}$. Two cases are possible:
\begin{itemize}
\item The geodesic starts at $r_0^{Sch}$ and reaches the horizon in a finite proper time. 
\item The geodesic starts from infinity and reaches the horizon in a finite time.  
\end{itemize} 
\item If $(E, \ell)\in\mathcal A_{circ}$, then the equation \eqref{cond::roots} admits one triple root if $\displaystyle E = \sqrt{\frac{8}{9}}$ or a double root. In this case, the geodesic is a circle. 
\end{enumerate}
\end{itemize}
\clearpage
\bibliographystyle{plain}
\bibliography{references}

\end{document}